\documentclass[11pt,reqno
]{amsart}
\usepackage{amsmath,amsfonts,amssymb,amscd,amsthm,amsbsy,amsxtra,bbm,bm, epsf,calc,comment}
\usepackage{color}
\usepackage{datetime}
\usepackage{latexsym}
\usepackage[english]{babel}
\usepackage{enumerate}
\usepackage{enumitem}
\usepackage{graphicx}
\usepackage{hyperref}
\usepackage{epsfig}

\def\e{\epsilon}
\def\R{{ \mathbb{R}}}

\def\N{\mathbb{N}}
\def\II{{\rm I\kern-0.5exI}}
\def\III{{\rm I\kern-0.5exI\kern-0.5exI}}

\newcommand{\norm}[1]{\lVert #1 \rVert}

\newcommand{\RR}{\mathbb{R}}

\DeclareMathOperator*{\argmin}{argmin}
\DeclareMathOperator*{\esssup}{ess\, sup}
\DeclareMathOperator*{\essinf}{ess\, inf}
\DeclareMathOperator*{\argmax}{argmax}

\DeclareSymbolFont{bbold}{U}{bbold}{m}{n}
\DeclareSymbolFontAlphabet{\mathbbold}{bbold}

\newcommand{\id}{id}

\newcommand{\spt}{\textup{spt}\,}

\numberwithin{equation}{section}
\newtheorem{theorem}{Theorem}[section]
\newtheorem{lemma}[theorem]{Lemma}
\newtheorem{prop}[theorem]{Proposition}
\newtheorem{corollary}[theorem]{Corollary}

\theoremstyle{remark}
\newtheorem{remark}[theorem]{Remark}

\theoremstyle{definition}
\newtheorem{definition}[theorem]{Definition}

\begin{document}

\title{Darcy's Law with a Source term}
\author[M. Jacobs]{Matt Jacobs}
\address{Department of Mathematics, UCLA, 520 Portola Plaza, Los Angeles, CA 90095, USA}
\email{majaco@math.ucla.edu}

\author[I. Kim]{Inwon Kim}
\address{Department of Mathematics, UCLA, 520 Portola Plaza, Los Angeles, CA 90095, USA}
\email{ikim@math.ucla.edu}

\author[J. Tong]{Jiajun Tong}
\address{Department of Mathematics, UCLA, 520 Portola Plaza, Los Angeles, CA 90095, USA}
\email{jiajun@math.ucla.edu}

\begin{abstract}
We introduce a novel variant of the JKO scheme to approximate Darcy's law with a pressure dependent source term.  By introducing a new variable that implicitly controls the source term, our scheme is still able to use the standard Wasserstein-2-metric even though the total mass changes over time.   Leveraging the dual formulation of our scheme,  we show that the discrete-in-time approximations satisfy many useful properties expected for the continuum solutions, such as a comparison principle and uniform $L^1$-equicontinuity.   Many of these properties are new even in the well-understood case where the growth term is absent.   Finally, we show that our discrete approximations converge to a solution of the corresponding PDE system, including a tumor growth model with a general nonlinear source term.
 
\end{abstract}

\maketitle

\section{Introduction}
\label{sec: introduction}

In this paper, we study Darcy's law with a pressure-dependent growth term, or more precisely the following equations
$$
\rho_t - \nabla \cdot(\rho \nabla p ) =\rho G(p,x)\quad \hbox{ and } p\in \partial s(\rho) \; \quad\hbox{ in } \R^d\times [0,T], \leqno (P)
$$
with initial data $\rho_0$. Here $\rho = \rho(x,t)$ represents the density of a flowing material, $p$ is the pressure generated by the internal energy
\begin{equation}\label{eq:energy}
E(\rho) =
\int_{\R^d} s(\rho(x)) \,dx,
\end{equation}
and $G$ is a pressure dependent growth term.  $(P)$ can be used to describe tumor growth models, where the malignant growth is limited only by the buildup of pressure when cells become too densely packed \cite{perthame2014}.  In order to capture this behavior, it is natural to restrict $s$ to be a convex, increasing and superlinear function.  Examples that we have in mind include the {\it R\'{e}nyi} energy given by
 $$
 s_m(\rho): =\frac{1}{m-1}\rho^m \,\,\hbox{ when } \rho\geq 0, \quad\hbox{ otherwise }+\infty.
 $$
 for $m>1$,
and its singular limit obtained as $m\to\infty$ given by
$$
 s_{\infty}(\rho): = 0 \,\,\hbox{ when } 0\leq \rho\leq 1,\quad\hbox{ otherwise }+\infty.
 $$
$s=s_{\infty}$ is a natural choice of internal energy for incompressible tumor growth models (see for instance \cite{PQV}).

\medskip

Our goal is to introduce a discrete-in-time variational scheme ({\it minimizing movements}) to approximate solutions of $(P)$. When the growth term $G$ is absent, the system of equations $(P)$ can be formally written as the gradient flow of the internal energy in 2-Wasserstein space (see \cite{otto_pme} for the case $s=s_m$ and \cite{maury,AKY14} for $s=s_{\infty}$).
With the presence of $G$, it is no longer clear whether the equation can be realized as a gradient flow.  First of all, in most scenarios, the energy $E$ increases along the flow.  As such, any gradient flow formulation must locate a different ``energy'' that is actually dissipated along the flow.     Secondly, one must also deal with the fact that the total mass of the density is not constant in time. This obstructs a straightforward application of a 2-Wasserstein minimizing movements scheme (a.k.a.\;the JKO scheme \cite{jko}), as the standard notion of optimal transport is only defined between densities with the same mass.  While it is possible to consider a modified version of Wasserstein distance to allow for changing mass (see for instance  \cite{glop_2019, chizat_unbalanced_interpolation, benamou_unbalanced}), the resulting gradient flow cannot capture the full generality of $(P)$.    For instance, in the case of the specific choice of $s=s_{\infty}$, the gradient flow formulation restricts the growth term to be linear and homogenous with respect to the pressure (see \cite{DMS, GLM}).  

\medskip

In this paper, we introduce a new discrete-in-time variational scheme for approximating the equation ($P$).
In contrast to previous results (e.g.\;\cite{liero2016optimal, DMS, GLM}), we do not modify the Wasserstein metric. Insted, we introduce an additional variational term that allows us to implicitly solve for the growth rate at each time step.  The advantage of this perspective is that our scheme can approximate any flow of the form ($P$).       Furthermore, the dual problem associated to our scheme has a very efficient numerical implementation using the recently introduced \emph{back-and-forth method} \cite{fast_ot, bfm_got}.   In particular, the numerical implementation via the back-and-forth method does not require introducing an additional time dimension, which allows for a faster computation time than schemes based around the Benamou-Brenier formula.  In addition, our scheme has no difficulty with the singular energy $s_{\infty}$ and produces a sharp boundary (see Proposition \ref{density_monotone}). 
\medskip

   In what follows, we will also show that our scheme captures many of the favorable properties of the underlying PDE ($P$), such as the comparison principle, finite propagation properties, and various uniform bounds,   as well as an $L^1$-equicontinuity property that generalizes the $BV$ bounds obtained in \cite{PQV}.  Let us emphasize that some of these properties for discrete-time solutions are new even when $G=0$.  Using these properties, we then show that the scheme converges to a solution of the continuum PDE ($P$) as the time step tends to zero.

\subsection{The discrete-in-time scheme}

The following discussion is a quick introduction of our scheme and its dual formulation. A detailed analysis, including  existence and uniqueness of the extreme values, will be delayed until Section \ref{sec:dual}.

\medskip

Given a time horizon $T>0$,  we choose a smooth, convex, and bounded domain $\Omega$ that is sufficiently large to contain the the flow in the time window $[0,T]$  (see Theorem \ref{main:3} and the discussion in the beginning of Section \ref{sec: equicontinuity} for more information on the choice of $\Omega$).   We then construct a discrete-in-time approximation to ($P$) as follows.   For a fixed time step size $\tau>0$, we define $\rho^{0,\tau} := \rho_0$ and then iterate the variational problem:
\begin{equation}\label{eq:mms}
(\rho^{n+1,\tau}, \mu^{n+1,\tau}):=\argmin_{\rho\in X, \mu\in \textup{AC}(\rho^{n,\tau})} \; J(\rho, \mu, \rho^{ n,\tau})
\end{equation}
where $X := \{\rho\in L^1(\Omega): E(\rho) <\infty\}$ and $AC(\rho)$ denotes the space of measures absolutely continuous with respect to $\rho$,
\begin{equation}\label{eq:primal_problem}
 J(\rho, \mu, \rho^{n,\tau}):=E(\rho)+ \tau F(\mu, \rho^{n,\tau})+\frac{1}{2\tau}W_2^2(\rho, \rho^{n,\tau}+\tau\mu),
\end{equation}
and
\begin{equation*}
F(\mu, \rho^{n,\tau}):=
\begin{cases}
\int_{\Omega} \rho^{n,\tau}(x) f\left(\frac{\mu(x)}{\rho^{n,\tau}(x)},x\right)\, dx & \textrm{if} \;\; \mu\in \textrm{AC}(\rho^{n,\tau}),\\
+\infty & \textrm{else.}
\end{cases}
\end{equation*}
Here $f=f(z,x)$ is the unique function defined by
\begin{equation}\label{f_def}
\left\{
\begin{array}{l}
f\big(G(0,x),x\big) = 0; \\
\partial_z f(z,x) =\{-b: z = G(b,x)\} \hbox{  when the set is nonempty, otherwise } +\infty.
\end{array}\right.
\end{equation}

\medskip

Our scheme differs from the usual JKO scheme due to the terms involving the variable $\mu$.  Indeed, $\mu(x)$ represents the amount of additional mass added at location $x$, and $f$ is a term that encourages growth at the locations where $\rho^{n,\tau}(x)\neq 0$.  Due to the growth term, we no longer expect to have the dissipation property $E(\rho^{n+1,\tau})\leq E(\rho^{n,\tau})$.  

We can recover the discrete analogue of the pressure variable in $(P)$ by applying convex duality to our scheme.  Indeed, the pressure at the $(n+1)$-th step solves the dual problem to \eqref{eq:mms},
\begin{equation}\label{eq:dual_scheme}
p_{n+1, \tau}\in\argmax_{p \in X^*} J^*(p,\rho^{n,\tau}).
\end{equation}

Here,
\begin{equation}\label{eq:dual_problem}
J^*(p,\rho^{n,\tau})=\int_{\Omega} \rho^{n,\tau}(x)\Big( p^c(x)+\tau \bar{G}\big(p^c(x),x\big)\Big)\, dx -\int_{\Omega} s^*(p(x)) dx,
\end{equation}
where $\partial_z \bar{G}(z,x)=G(z,x)$, $s^*$ is the Legendre transform of $s$, and
\begin{equation*}
p^{c}(x):=\inf_y p(y)+\frac{1}{2\tau}|y-x|^2,
\end{equation*}
is the quadratic {\it $c$-transform} which plays an essential role in optimal transport. In Section \ref{sec:dual}, we will derive the connection between the primal and dual problems, including the relation $p_{n+1}\in \partial s(\rho_{n+1})$  (see Proposition~\ref{primal_dual}).

Much of our subsequent analysis will focus on the dual problem \eqref{eq:dual_scheme}, which in many ways is easier to study than the primal problem.  This is due to the fact that variations of the $c$-transform are easier to study than variations of the 2-Wasserstein distance (which essentially requires introducing a dual variable anyway).   Of course, one could have also chosen problem \eqref{eq:dual_scheme} as the starting point for the scheme, however, the physical interpretation of the primal problem is much clearer than that of the dual problem.    Note that when $s$ and $s^*$ are differentiable, the second condition in $(P)$ yields
\begin{equation*}
 p = s'(\rho), \quad \rho = (s^*)'(p).
 \end{equation*}
   Hence $(P)$ can be written in a weak form as
\begin{equation}\label{weak00}
((s^*)'(p))_t - \Delta s^*(p)=  \rho G(p,x).
\end{equation}
This is a nonlinear parabolic equation in terms of $p$, whose particular structure is discussed in the classical paper \cite{Alt_Luckhaus} .   This perspective further clarifies why it is easier to work with the dual problem.  Indeed, many of the beneficial properties that we develop from the dual problem are related to the parabolic structure of the pressure equation.

\subsection{Assumptions and Main results}

 Throughout the paper, we require the energy density function $s:\RR\to \RR\cup\{+\infty\}$ to be proper, convex, lower semi-continuous, superlinear, and satisfy
$$
s(y)=+\infty\hbox{ if } \{y<0\},\;  s(0)=0, \; s(\rho)\mbox{ is increasing in }[0,\infty), \; \textrm{and} \; \lim_{z\to 0^+}\frac{s(z)-s(0)}{z}=0.
$$
The first condition forces the density function to be nonnegative, while the second and third conditions, $s(0)=0$ and $s(\rho)$ increasing, are physically natural for tumor growth models.  In addition, these two conditions ensure that solutions stay compactly supported if initially so (see Section \ref{sec: finite speed propagation}).  The last condition is only for notational convenience --- this condition ensures that the density variable is positive whenever the pressure is positive.  Without this assumption, one would have to constantly refer to pressure values in the set $\partial s((0,\infty))$, which quickly gets quite cumbersome.

As for $G$ we assume the following conditions:
\begin{enumerate}[label=(G\arabic*)]
\item \label{assumption: growth at zero pressure} $G(0,x)$ is strictly positive for all $x$.
\item\label{assumption: regularity} $G(z,x)$ is Lipschitz continuous with respect to $(z,x)\in\RR\times\R^d$ and decreasing with respect to $z$.
\item\label{assumption: death when large pressure} for all $x$ there exists $b(x)$ such that $ G(b(x),x)=0$, and $0 \leq b_0 \leq b(x) \leq b_1<\infty$ for all $x\in\R^d$.
\item\label{assumption: bound} $B:=\sup_{(z,x)\in\RR^+\times \R^d} |G(z,x)|<\infty$ .
\end{enumerate}
Assumptions \ref{assumption: growth at zero pressure}-\ref{assumption: death when large pressure} are physically natural and correspond to the assumptions that growth occurs when the pressure is zero,  growth slows continuously as the pressure increases, and at each location there is a threshold value where growth will cease if the pressure becomes too high. \ref{assumption: bound} is a technical condition bounding the growth rate, which will streamline our subsequent analysis. This condition could almost certainly be weakened, for instance to local bounds within the range of pressure. However, in the context of tumor growth models, we do not believe unbounded growth is relevant enough to justify the additional complication.
\medskip

To approximate the equation ($P$), we define the piecewise-constant-in-time interpolations
\begin{equation}
\label{interpolation}
\begin{split}
\rho^{\tau}(x,t)&:=\rho^{n+1, \tau}(x)\quad \textrm{if} \; t\in [n\tau, (n+1)\tau),\\
\mu^{\tau}(x,t)&:=\mu^{n+1,\tau}(x)\quad \textrm{if} \; t\in [n\tau, (n+1)\tau),\\
p^{\tau}(x,t)&:=p_{n+1,\tau}(x)\quad \textrm{if} \; t\in [n\tau, (n+1)\tau),
\end{split}
\end{equation}
starting with some given nonnegative initial data $\rho^{0,\tau} = \rho_0$. We assume that $\rho_0$ is compactly supported with
\begin{equation}\label{good_data_0}
\inf \partial s(\hat{M}_0) <\infty \hbox{ where } \hat{M}_0:= \norm{\rho_0}_{\infty}.
\end{equation}
Note that in particular $\hat{M}_0<\infty$.

\medskip

We first point out that the choice of $\Omega$, if sufficiently large for the given time range, does not affect the discrete-time solutions. Indeed, if $\rho_0$ is compactly supported and bounded, then $(\rho^{\tau}, \mu^{\tau}, p^{\tau})$ stay compactly supported and bounded (see Theorem~\ref{main:3}
and the discussion in the beginning of Section \ref{sec: equicontinuity}).

\medskip

In the next two theorems, we will denote $Q:=   \R^d\times [0,T]$, with $\Omega$ sufficiently large . We will then show that $(\rho^\tau,\mu^\tau,p^\tau)$ converge to a continuum solutions of $(P)$ as  we take $\tau\to 0$. To be more precise, we have
\begin{theorem}\label{main:1}
Let $M_0:= \max(b_1,\inf \partial s(\norm{\rho_0}_{L^\infty(\R^d)}))$, where $b_1$ is given in \ref{assumption: death when large pressure}.
Suppose either $s \in C^1_{loc}([0,\infty))$ or $G(\cdot,x)$ is affine in $[0, M_0]$ for all $x\in \R^d$. Then for any $T>0$,
\begin{itemize}
\item[(a)] For all $\tau \leq (2B)^{-1}$, $\rho^{\tau}, \mu^{\tau}, p^{\tau}_+$ are uniformly bounded in $L^{\infty}(Q)$.
\end{itemize}
There exists $\rho, p \in L^{\infty}(Q)$ such that, as $\tau \to 0$, up to a subsequence,
\begin{itemize}
\item[(b)] $\mu^{\tau}\rightharpoonup \rho G(p,x)$ in $L^1([0,T]; W^{-1,1}(\R^d))$;
\item[(c)] $\rho^{\tau}\to \rho$ in $L^1(Q)$;
\item[(d)] $p^{\tau}_+ \rightharpoonup p$ and $s^*(p^{\tau})\rightharpoonup s^*(p)$ in $L^1(Q)$. They also a.e.\;converge provided $s \in C^1_{loc}([0,\infty))$.
\end{itemize}
Moreover,
\begin{itemize}
\item[(e)] $(\rho,p)$ is a very weak solution of $(P)$ in the sense that $p \in \partial s(\rho)$ a.e.\;and
$$
\int_0^{t_0}\int_{\R^d}\rho \phi_t + s^*(p)\Delta \phi  + G(p,x)\rho\phi\, dxdt=  \int_{\R^d} (\rho\phi)(x,t_0)-(\rho\phi)(x,0)\, dx,
$$
for any $\phi\in C^{\infty}(Q)$ and for a.e.\;$t_0\in [0,T]$.
\item[(f)] If $\rho_0 \in BV$, we also have $\rho(t,\cdot)\in BV$ with its BV norm growing at most exponentially in time.
\end{itemize}
\end{theorem}

When $s=s_{\infty}$, despite the irregular nature of the energy functional, strong monotonicity properties holds for $\rho^{\tau}$ and $p^{\tau}$, which leads to convergence results. In this case the initial data we consider is a compactly supported function $\rho_0$ in $\R^d$ with $\rho_0\leq 1$.

\begin{theorem}\label{main:2} Let $s=s_{\infty}$. For any $T>0$, there exists $\rho\in L^{\infty}(Q)$ and $p\in L^2_{loc}([0,T];H^1_{loc}(\R^d))$ such that $(a),(b)$ and $(f)$ in Theorem~\ref{main:1}.
Moreover,
\begin{itemize}
\item[(c')] $\rho^{\tau}\in[0, 1]$ is monotone increasing in time, and converges to $\rho$ in $L^1(Q)$. If $\rho_0\in \{0,1\}$ a.e., then $\rho^{\tau}, \rho \in \{0,1\}$ a.e..
\item[(d')]  $p^{\tau}_+$ is monotone increasing in time, and converges to $p$ in $L^2(Q)$. Moreover $\nabla p^{\tau}_+\rightharpoonup \nabla p$ in $L^2(Q)$.
\item[(e')] $(\rho,p)$ is a weak solution of $(P)$ in the sense that $p(1-\rho)=0$ a.e.\;and
  \begin{equation*}
 \int_0^{t_0} \int_{\R^d} \rho\partial_t \phi -  \nabla p\cdot\nabla \phi +G(p,x)\rho\phi\, dx dt=  \int_{\R^d} (\rho \phi)(x,t_0) dx - (\rho\phi)(x,0)\, dx,
 \end{equation*}
 for any $\phi\in C^{\infty}(Q)$ and for a.e.\;$t_0\in [0,T]$.
 \end{itemize}
 \end{theorem}

One of the key ingredients in establishing the above results is the comparison principle among the discrete solutions, which is of independent interest. Similar results have been obtained for the case of $E=E_m$ in \cite{AKY14}, but our argument generalizes and simplifies the original proof by arguing through the dual formulation of the scheme. As a consequence, we can construct barriers to show the density propagates with finite speed.   In particular, this allows us to ensure that the evolution remains bounded in any finite time horizon.    Another main ingredient used to establish the convergence result is the spatial equicontinuity of the density variables. The equicontinuity estimate is motivated from \cite{JKT}, but requires a substantially different argument in this setting due to the presence of the growth term.  We collect these results in the following theorem.

\begin{theorem}\label{main:3}
\begin{itemize}
\item[(a)]{\rm(Comparison principle)} If $(\rho_1)^{n,\tau} \leq (\rho_2)^{n,\tau}$ both satisfying \eqref{good_data}, then $(\rho_1)^{n+1,\tau} \leq (\rho_2)^{n+1,\tau}$ and $(p_1)_{n+1,\tau} \leq (p_2)_{n+1, \tau}.$
\item[(b)]{\rm (Finite-speed propagation)} If $\rho_0$ is supported in $B_{R_0}$ satisfying \eqref{good_data_0}, then there exists $R_1,R_2>0$ indepedent of $\tau$ such that
$$
\spt \rho^\tau (\cdot,t) \subset B_{R_0+ R_1+ R_2t}.
$$
\item[(c)] {\rm ($L^1$-equicontinuity)} For any $y\in \R^d$ and sufficiently small $\tau$ we have
$$
\lim_{\e\to 0 } \int_0^T \int_{\R^d} |\rho^{\tau}(x+\e y,t) - \rho^{\tau}(x,t)| dx dt = 0.
$$
 \end{itemize}
\end{theorem}
Further characteristics of $R_1,R_2$ and its dependence on $\rho_0$, $G$ and $s$ are given in Section \ref{sec: finite speed propagation} as well as in the Appendix \ref{appendix: improved barriers}.

\medskip

Lastly, we briefly study coincidence of our continuum solutions with other notions of solutions.
Both weak solutions and viscosity solutions approach are available for the well-posedness of the tumor model with $s=s_{\infty}$ \cite{KP,MPQ,PQV,AKY14,CKY}. There its well-posedness and coincidence are established, as well as its characterization as the limit of weak solutions with $s=s_m$ as $m\to\infty$.  Given the extensive analysis on the continuum solutions in aforementioned references, we do not pursue a qualitative analysis at the continuum level.

\begin{theorem}[Coincidence]
\label{thm: coincidence}

\begin{itemize}
\item[(a)] Suppose $s\in C^1_{loc}([0,\infty))$. Then under a condition \eqref{extra} that includes $s=s_m$ for $1<m<\infty$, the continuum pair $(\rho,p)$ obtained in Theorem~\ref{main:1} is the unique weak solution of $(P)$.  In particular they are the limit of the entire sequence $(\rho^{\tau}, p^{\tau})$ as $\tau\to 0$.
\item[(b)] When $s=s_{\infty}$ the pair $(\rho,p)$ obtained in Theorem~\ref{main:2} coincides with the unique weak solution obtained in \cite{PQV}. In particular they are the limit of the entire sequence $(\rho^{\tau}, p^{\tau})$ as $\tau\to 0$.
 \end{itemize}
  \end{theorem}

\subsection{Organization of the paper} The remainder of the paper is organized as follows.
In section \ref{sec:background}, we recall basic properties of optimal transport and convex duality. In Section \ref{sec:dual}, we develop properties of the primal and dual variational problems.  In particular, we show that the primal and dual problems are linked by strong duality, and we establish uniform bounds for discrete densities. In Section \ref{sec: comparison}, we show the comparison principle, Theorem~\ref{main:3}(a), based on properties of the dual problem.  This generalizes and simplifies the comparison principle proof developed in \cite{AKY14}.  Next, Section \ref{sec: finite speed propagation} establishes  the finite propagation property, Theorem~\ref{main:3}(b), based on the comparison principle and barrier constructions. Interestingly, the barriers are constructed using the dual problem and are backwards-in-time, which is natural in view of the duality approach.  The last two sections, Sections \ref{sec: equicontinuity} and \ref{sec: continuum limit},  establish the compactness properties of discrete solutions and then show their convergence to the continuum limit as $\tau$ tends to zero. Section \ref{sec: equicontinuity} focuses on strong compactness of the density variable in $L^1$. The arguments in this section follow the ideas from \cite{JKT}, however, we need to introduce significantly new ingredients, as the growth term prevents the use of $L^1$-contraction argument from \cite{JKT}. Section \ref{sec: continuum limit} collects the results from the previous sections to derive the main convergence theorems Theorem~\ref{main:1} and Theorem~\ref{main:2}.
Section \ref{sec: coincidence of solutions} discusses coincidence of our weak solutions with other existing notions of solution stated in Theorem \ref{thm: coincidence}.
Finally, we construct in the Appendix \ref{appendix: improved barriers} a more refined version of barriers from Section \ref{sec: finite speed propagation}, which give a finer characterization of the propagation of the density support.

\subsection{Acknowledgement} M.J. is  supported by ONR N00014-18-1-2527 and AFOSR MURI FA9550-18-1-0502. I.K is supported by NSF grant DMS-1900804 and the Simons Fellowship.

\section{Preliminary Results}\label{sec:background}

We begin with recalling some essential properties of optimal transport and dual functions. Since we primarily work with optimal transport in its dual formulation, we shall work extensively with the $c$-transform.  Here we focus on the specific cost $c(x,y) := \frac{|x-y|^2}{2\tau}$ for some $\tau>0$. We follow the notations given in \cite{JKT}.

\begin{definition} Given a function $p:\Omega\to\RR$ the {\it $c$-transform} of $p$ is given by
\begin{equation*}
p^{c}(y)=\inf_{x\in\Omega} p(x)+c(x,y).
\end{equation*}
Given a function $q:\Omega\to\RR$  the {\it  conjugate $c$-transform} is given by
\begin{equation*}
q^{\bar{c}}(x):=\sup_{y\in \Omega} q(y)-c(x,y).
\end{equation*}
\end{definition}

\begin{lemma}[\cite{otam}]\label{lem:ccc}
Given functions $p, q:\Omega\to\RR$, we have
\[
p^{c\bar{c}}\leq p, \quad q\leq q^{\bar{c}c},
\]
and
\[
p^{c\bar{c}c}=p^{c},\quad q^{\bar{c}c\bar{c}}=q^{\bar{c}}.
\]
\end{lemma}

\begin{definition}
We say that a function $p:\Omega\to \RR$ is {\it $c$-concave} if $p^{c\bar{c}}=p$, and we say a pair of functions $p, q:\Omega\to\RR$ are {\it $c$-conjugate} if $p^c=q$ and $q^{\bar{c}}=p$.
\end{definition}

The following regularity result is a well-known consequence of the $c$-transform definition.

\begin{lemma}\label{lem:lip}
If $p$ is $c$-concave, then $p$ is Lipschitz and the Lipschitz constant depends only on $c$ and $\Omega$.
\end{lemma}

The following two lemmas establish the fundamental relationship between optimal transport and the $c$-transform.
\begin{lemma}\label{lem:c_duality}  If $\mu$ is a nonnegative measure, then for any bounded function $p:\Omega\to \RR$,
\[
\inf_{\rho\in L^1(\Omega), \rho(\Omega)=\mu(\Omega)}\, \int_{\Omega}p(x)\rho(x)\, dx+ \frac{1}{2\tau}W_2^2(\rho, \mu)=\int_{\Omega} p^c(y)\, d\mu(y).
\]
\end{lemma}

\begin{lemma}[\cite{gangbo_polar, gangbo_thesis, gangbo_mccann}]
\label{gangbo}
Let $p:\Omega\to\RR$ be $c$-concave with $c(x,y) = \frac{|x-y|^2}{2\tau}$. Define $T_p:\Omega\to\Omega$ be the unique solution to
\begin{equation}\label{eq:forward}
T_p(y)=y-\tau\nabla p^c(y).
\end{equation}
Then $T_p$ is invertible a.e., and $T_p^{-1}$ is the unique solution to
\begin{equation*}
T_p^{-1}(x)=x+\tau\nabla p(x).
\end{equation*}

Moreover, if $\mu$ is a nonnegative measure and if $\phi:\Omega\to\RR$ is a continuous function, then
\begin{equation*}
\lim_{t\to 0^+} \int_{\Omega} \frac{(p+t\phi)^c(y)-p^c(y)}{t}\,d\mu(y)=\int_{\Omega} \phi(T_{p}(y))\,d\mu(y)
\end{equation*}
\end{lemma}

\begin{remark}\label{dual2}
The maps $T_p$ and $T_p^{-1}$ can additionally be characterized as the unique solutions to the optimization problems:
\begin{equation*}
T_p(y)=\argmin_{x\in\Omega} p(x)+c(x,y),\quad T_p^{-1}(x)=\argmax_{y\in\Omega} p^c(y)-c(x,y).
\end{equation*}
\end{remark}

Now we can finally state the fundamental result guaranteeing the existence and uniqueness of optimal transport maps.
\begin{theorem}[\cite{brenier_polar,gangbo_thesis,gangbo_mccann}]
\label{thm:fund_ot} If $\mu, \nu\in L^1(\Omega)$ are nonnegative densities with the same mass, then there exists a $c$-concave function $p^*:\Omega\to\RR$ such that
\[
p^*\in\argmax_p \int_{\Omega} p^c(y) \mu(y)\, dy-\int_{\Omega} p(x)\nu(x)\, dx,
\]
\[
W_2^2(\mu,\nu)=\int_{\Omega} (p^*)^c(y) \mu(y)\, dy-\int_{\Omega} p^*(x)\nu(x)\, dx,
\]
Moreover $T_{p^*}$ is the unique optimal map transporting $\mu$ to $\nu$, and $T_{p^*}^{-1}$
is the unique optimal map transporting $\nu$ to $\mu$.  Conversely, if $\tilde{p}$ is a $c$-concave function such that $T_{\tilde{p}\,\#}\mu=\nu$ then $T_{\tilde{p}}$ is the unique optimal map transporting $\mu$ to $\nu$ and $T_{\tilde{p}}^{-1}$ is the unique optimal map transporting $\nu$ to $\mu$.
\end{theorem}

We conclude this section with some results from convex duality theory that we will use extensively in our arguments.

\begin{lemma}[\cite{JKT}]\label{dual_relation}
For any proper, lower semi-continuous, convex function $h:\RR\to\RR\cup\{+\infty\}$, we have $p\in \partial h(z)$ if and only if $pz=h(z)+h^*(p)$.  Here $h^*$ is the convex dual of $h$ defined by $s^*(p) := \sup_{\rho \in \RR}\{ \rho p - s(\rho)\}$.
\end{lemma}

\begin{lemma}[\cite{JKT}]\label{lem:s_star_increasing}
Suppose $h:\RR\to\RR\cup\{+\infty\}$ is  proper, lower semi-continuous, convex, and $h(z) \equiv +\infty$ if $z<0$.
Then $h^*$ is increasing, and it is strictly increasing on $\partial h((0,\infty))$.
\end{lemma}

\section{Properties of the Primal and Dual Problems}\label{sec:dual}

In this section, we study properties of solutions $(\rho^{n,\tau}, \mu^{n,\tau})$ to the primal problem \eqref{eq:mms}, and their relationship to the dual pressure variables $p_{n,\tau}$ that maximize the dual problem \eqref{eq:dual_scheme}. Since $\tau$ is fixed for the results of this section, we denote $\rho^{n}:= \rho^{n,\tau}$ for simplicity.

\medskip

\subsection{Equivalence and well-posedness}

We begin by showing the existence of a unique minimizer for the primal problem \eqref{eq:primal_problem}. Recall that $X:=\{\rho\in L^1(\Omega): E(\rho)<\infty\}$. For $\rho \in X$ and $\Sigma\subset\Omega$, define
\begin{equation*}
\rho(\Sigma):= \int_{\Sigma} \rho\, dx.
\end{equation*}
Define the dual energy $E^*: X^* \to \R$ such that
\begin{equation*}
E^*(p):= \int_{\Omega} s^*(p(x)) dx, \quad s^*(p):= \sup_{y\in \RR} \{py - s(y)\},
\end{equation*}
where $X^*$ is the convex dual of $X$ with respect to $E$, namely
$$
X^*:=\{p:\Omega\to [-\infty,+\infty]:\;p\;\mathrm{is\;measurable},\;E^*(p)<+\infty\}.
$$

We begin with the following simple lemma, which establishes weak duality between the primal and dual problems.
\begin{lemma}\label{derivation}
Suppose $\rho^n \in X$. Then
$$
\inf_{(\rho,\mu)\in X\times AC(\rho^n)} J(\rho, \mu,\rho^n)\geq \sup_{p\in X^*}J^*(p,\rho^n).
$$
\end{lemma}

\begin{proof}
The energy $E$ is convex, proper and lower semi-continuous, so
\[
E(\rho)=(E^*)^*(\rho)=\sup_{p\in X^*} (\rho, p)-E^*(p).
\]
It is thus immediate that the primal problem \eqref{eq:mms}  is equal in value to the primal-dual problem
\begin{equation*}
\inf_{\rho\in X,\, \mu\in \textup{AC}(\rho^n)}\sup_{p\in X^*} \Big ( (\rho, p)+\frac{1}{2\tau}W_2^2(\rho, \rho^n+\tau\mu)-E^*(p)+\int_{\Omega} \tau\rho^n(x)f\Big(\frac{\mu(x)}{\rho^n(x)},x\Big) \Big ).
\end{equation*}

By switching the $\inf$ and $\sup$, the value only decreases, and after further enlarging the search space for $\rho$ the above quantity is bounded from below by
\[
\sup_{p\in X^*} \inf_{\rho\in L^1(\Omega),\, \mu\in \textup{AC}(\rho^n)}\Big( (\rho, p)+\frac{1}{2\tau}W_2^2(\rho, \rho^n+\tau\mu)-E^*(p)+\int_{\Omega} \tau\rho^n(x)f\Big(\frac{\mu(x)}{\rho^n(x)},x\Big)\Big),
\]
which, by Lemma~\ref{lem:c_duality},
\[
\sup_{p\in X^*} \inf_{ \mu\in \textup{AC}(\rho^n)} (\rho^n+\tau\mu, p^c)-E^*(p)+\int_{\Omega} \tau\rho^n(x)f\Big(\frac{\mu(x)}{\rho^n(x)},x\Big),
\]
 Now it is clear that $\mu$ should be chosen so that
\begin{equation*}
p^c(x)\in -\partial_z f\Big(\frac{\mu(x)}{\rho^n(x)},x\Big),
\end{equation*}
which, by \eqref{f_def}, is equivalent to
\begin{equation*}
\mu(x) = \rho^n(x)G(p^c(x),x).
\end{equation*}
If we plug in this choice, we obtain the maximization problem
\begin{equation*}
\sup_{p\in X^*}  \int_{\Omega}\rho^n(x)\bigg( p^c(x)+  \tau p^c(x)G\big(p^c(x),x) +\tau  f\Big(G\big(p^c(x),x\big),x\Big)\bigg)\, dx -E^*(p).
\end{equation*}
Note that from \eqref{f_def}
\[
\partial_z \Big( zG\big(z,x) +  f\big(G\big(z,x\big),x\big)\Big)=G(z,x)
\]
and thus
\begin{equation}\label{obs1}
zG(z,x) +  f\big(G (z,x),x\big)=\bar{G}(z,x),
\end{equation}
and we conclude.
\end{proof}

Next we upgrade the previous proposition and show that the primal and dual problems satisfy a strong duality principle.  This guarantees that the problems attain the same optimal value and links the primal and dual variables through necessary optimality conditions.

\begin{prop}\label{primal_dual}
Suppose that $\rho^n \in X$ and satisfies
\begin{equation}\label{good_data}
0= \limsup_{b\to-\infty} \partial s^*(b)  <\frac{1}{|\Omega|}\int_{\Omega} \rho^n  dx < \liminf_{b\to\infty} \partial s^* (b).
\end{equation}
Then the primal problem \eqref{eq:mms} has a unique minimizer $(\rho^{n+1},\mu^{n+1})\in X\times AC(\rho^n)$ and the dual problem \eqref{eq:dual_scheme} has a $c$-concave maximizer $p_{n+1}\in X^*$ and
\[
\inf_{(\rho,\mu)\in X\times AC(\rho^n)} J(\rho, \mu,\rho^n)= \sup_{p\in X^*}J^*(p,\rho^n).
\]
Moreover, for a.e.\;$x\in\Omega$,
\begin{equation}\label{eq:pd_relation_energy}
p_{n+1}\in \partial s(\rho^{n+1}), \quad \rho_{n+1}\in \partial s^*(p^{n+1}),
\end{equation}
\begin{equation}\label{eq:pd_relation_pushforward}
T_{p_{n+1}\, \#} (\rho^n+\tau\mu^{n+1})=\rho^{n+1} \hbox{ with } T_{p_{n+1}}^{-1} (x) = x + \tau \nabla p_{n+1},
\end{equation}
and
\begin{equation}\label{eq:pd_relation_growth}
\mu^{n+1}(x)=\rho^n(x) G(p^c_{n+1}(x),x).
\end{equation}

 \end{prop}

 \begin{remark}
In this paper, we will mostly concern ourselves with compactly supported initial data in $\R^d$.  In this case, we can choose an arbitrarily large domain where \eqref{good_data} always holds for any given time range (see Corollary~\ref{cor:domain}).  However, if one wishes to consider a version of the problem where the density is restricted  $\Omega$, \eqref{good_data} can only fail in trivial cases.  Indeed, if \eqref{good_data} does not hold, then $\rho^n$ must already be the minimizer of the primal problem (i.e.\;$\rho^{n+1}=\rho^n$), and so the evolution has already reached a stationary state.
\end{remark}

\begin{proof}
 Let $J^*$ be as given in \eqref{eq:dual_problem}. Given some function $p:X^*\to \RR$, we can compute
\[
J^*(p^{c\bar{c}},\rho^n)=\int_{\Omega} \rho^n(x)\Big( p^{c\bar{c}c}(x)+\tau \bar{G}(p^{c\bar{c}c}(x),x)\Big)\, dx -E^*(p^{c\bar{c}}).
\]
From Lemmas \ref{lem:ccc} and \ref{lem:s_star_increasing}, we have
\[
\int_{\Omega} \rho^n(x)\Big( p^{c\bar{c}c}(x)+\tau \bar{G}(p^{c\bar{c}c}(x),x)\Big)\, dx=\int_{\Omega} \rho^n(x)\Big( p^{c}(x)+\tau \bar{G}(p^{c}(x),x)\Big)\, dx
\]
and
\[
E(p^{c\bar{c}})\leq E(p).
\]
Thus $J(p^{c\bar{c}},\rho^n) \geq J(p,\rho^n)$ and we have
\[
\sup_{p\in X^*} J^*(p,\rho^n)= \sup_{p\in X^*,\, p^{c\bar{c}}=p} J^*(p,\rho^n),
\]
which allows the search to be restricted to the space of $c$-concave functions.

Let $p_k$ be a sequence of bounded $c$-concave functions such that
\[
 \lim_{k\to\infty}   J^*(p_k,\rho^n)=\sup_{p\in X^*,\, p^{c\bar{c}}=p}  J^*(p,\rho^n).
\]
If we set $\alpha_k=\frac{1}{|\Omega|}\int_{\Omega} p_k(x)\, dx$, then $\tilde{p}_k=p_k-\alpha_k$ is $c$-concave and has zero mean.  Thanks to Lemma \ref{lem:lip}, it follows that $\tilde{p}_k$ is uniformly bounded in $W^{1,\infty}(\Omega)$.  Thus, we can assume without loss of generality that $\tilde{p}_k$ converges uniformly to a zero mean function $\tilde{p}$.  Next, we chose
\[
\beta_k\in \argmax_{\beta\in [-\infty,\infty]} F(\beta):= J^*(\tilde{p}_k+ \beta,\rho^n).
\]
Since $(\tilde{p}_k(x)+\beta)^c=\tilde{p}_k^c(x)+\beta$ for any $\beta\in [-\infty, \infty]$, we see that
$$
 F'(\beta) = \int_{\Omega} \rho^n(x) (1+ \tau G(\tilde{p}_k^c(x)+\beta,x) )\, dx - \int_{\Omega}\partial s^*(\tilde{p}_k(x)+\beta,x)\, dx,
$$
which decreases with respect to $\beta$. Therefore $J^*(\tilde{p}_k+\beta, \rho^n)$ is concave with respect to $\beta$. Since $\tilde{p}_k$ is uniformly bounded, the assumption \eqref{good_data} yields some $M>0$ that depends on $\rho^n$, $G$, $c$ and $\Omega$, such that $F'$ is negative if $\beta>M$ and is positive if $\beta<-M$.
Hence $\beta_k$ exists and must be bounded uniformly in $\RR$.  Hence, we can assume without loss of generality that the $\beta_k$ converge to a finite limit $\tilde{\beta}$.

Define $p^*:=(\tilde{p}+\tilde{\beta})^{c\bar{c}}$.  We then have
\[
J^*(p^*,\rho^n)\geq  J^*(\tilde{p}+\tilde{\beta},\rho^n)\geq \limsup_{k\to\infty}  J^*(\tilde{p}_k+\beta_k,\rho^n)
\]
where the last inequality follows from the fact that the $c$-transform, $G(\cdot,x)$, and $-s^*$ are upper semi-continuous with respect to pointwise convergence.
Thanks to our choice of $\beta_k$, we see that
\[
\limsup_{k\to\infty}  J^*(\tilde{p}_k+\beta_k,\rho^n)\geq \limsup_{k\to\infty}  J^*(\tilde{p}_k+\alpha_k,\rho^n)=\sup_{p\in X^*,\, p^{c\bar{c}}=p}  J^*(p,\rho^n).
\]
Therefore, we can conclude that $p^*$ is a $c$-concave maximizer of the dual problem.

\medskip

Define
\begin{equation}\label{def010}
\mu^*(x):= \rho^n(x)G((p^*)^c(x),x), \hbox{ and } \rho^*:=T_{p^*\, \#}(\rho^n + \tau\mu^*).
\end{equation}
We would like to show that $(\rho^*, \mu^*)$ minimizes $J(\rho,\mu, \rho^n)$.

Using Lemma \ref{gangbo},  the optimality condition for $p^*$ implies that there exists $\omega\in\partial s^*(p^*)$ such that for every continuous function $\phi:\Omega\to\RR$
\[
\int_{\Omega} \phi(T_{p^*}(y))(1+\tau G((p^*)^c(x),x))\rho^n(y)\, dy-\int_{\Omega} \omega(x)\phi(x)\, dx=0.
\]
Thus, we must have
\begin{equation}\label{def020}
\rho^*\in \partial s^*(p^*) \hbox{ a.e. in } \Omega.
\end{equation}
  Therefore, from Lemma \ref{dual_relation} we have the duality relation

\begin{equation}\label{obs2}
\int_{\Omega} \rho^*(x)p^*(x)\, dx=E(\rho^*)+E^*(p^*).
\end{equation}
From \eqref{obs1} and \eqref{obs2}, as well as the definition of $\mu^*$ in \eqref{def010}, we have
\[
\begin{split}
J^*(p^*, \rho^n) - \int_{\Omega} \tau \rho^n f\left(\frac{\mu^*}{\rho^n},x\right)\, dx&=\int_{\Omega} (p^*)^c(y)(\rho^n + \tau\mu^*)(y)\, dy -\int_{\Omega} p^*(x)\rho^*(x)\, dx+ E(\rho^*)\\
&= \frac{1}{2\tau}W_2^2(\rho^*,\rho^n +\tau\mu^n) + E(\rho^*),
\end{split}
\]
where the last equality follows from Theorem \ref{thm:fund_ot}.  This allows us to conclude that
\[
J^*(p^*,\rho^n)= J(\rho^*, \mu^n,\rho^n)\geq \inf_{\rho\in X} J(\rho, \mu,\rho^n).
\]
 On the other hand we have, from Lemma \ref{derivation},
\[
\inf_{(\rho,\mu)\in X\times AC(\rho^n)} J(\rho, \mu,\rho^n)\geq \sup_{p\in X^*}J^*(p,\rho^n)=J^*(p^*,\rho^n).
\]
Therefore, it follows that $(\rho^*,\mu^*)$ is a minimizer of the primal problem.  Denoting $p^* = p_{n+1}$, $\rho^*= \rho^{n+1}$ and $\mu^*=\mu^{n+1}$, \eqref{eq:pd_relation_energy}-\eqref{eq:pd_relation_growth} follow from \eqref{def010}, \eqref{def020} and Lemma~\ref{gangbo}.

\medskip

Now let us establish uniqueness of $(\rho^*, \mu^*)$.  Suppose that we have two minimizers  $(\rho_0, \mu_0)$ and $(\rho_1, \mu_1)$.   Let $T_i$ be the optimal transport map from $\rho^n+\tau \mu_i$ to $\rho_i$ (note that $T_i$ must exist since we have established that minimizers are absolutely continuous with respect to the Lebesgue measure).
 Let $\pi_i:\Omega\times\Omega\to [0,\infty)$ be the transportation plan associated to the map $T_i$, and define
 $$
 (\rho_t, \mu_t, \pi_t):=t(\rho_0,\mu_0, \pi_0)+(1-t)(\rho_1, \mu_1, \pi_1) \hbox{ for } t\in [0,1].
 $$

 By convexity, $(\rho_t,\mu_t)$ must be a minimizer for all $t\in [0,1]$,  hence $t\mapsto J(\rho_t, \mu_t, \rho^n)$ must be constant.
  If $\pi_t$ is not an optimal transportation plan for some $t\in (0,1)$, then
 \[
 W^2_2(\rho_t, \rho^n+\tau\mu_t)<t W_2^2(\rho_1, \rho^n+\tau\mu_1)+(1-t) W_2^2(\rho_0, \rho^n+\tau\mu_0),
 \]
 and thus,
 \[
 J(\rho_t, \mu_t, \rho^n)<tJ(\rho_1, \mu_1, \rho^n)+(1-t)J(\rho_0, \mu_0, \rho^n)
 \]
 contradicting the optimality of $(\rho_1, \mu_1)$ and $(\rho_0,\mu_0)$.   Therefore, $\pi_t$ must be an optimal plan for all $t\in [0,1]$.  $\rho_t$ and $\rho^n+\tau\mu_t$ are absolutely continuous, so $\pi_t$ must be induced by a map $T_t$ pushing $\rho^n+\tau\mu_t$ to $\rho_t$.    For $x\in \spt\rho^n$  we have
\[
\pi_t(x,y)=t\delta(T_0(x)-y)+(1-t)\delta(T_1(x)-y)=\delta(T_t(x)-y),
\]
which is only possible if there is a single map $T$ such that $T(x)=T_0(x)=T_1(x)$ for almost all $x$ in $\spt\rho^n$.
 It remains to show that $\mu_1=\mu_0$, which would yield that $\rho_1=T_{\#}(\rho^n+\tau\mu_1)=T_{\#}(\rho^n+\tau\mu_0)=\rho_0$.

 \medskip

 Once again using the fact that $t\mapsto J(\rho_t, \mu_t, \rho^n)$ is constant, we can conclude that
 for all $t\in (0,1)$ there exists $\eta_t(x)\in \partial_z f(\frac{\mu_t(x)}{\rho^n(x)},x)$ and $\zeta_t\in \partial s(\rho_t)$ such that
\[
(\eta_t, \mu_1-\mu_0)+(\zeta_t, \rho_1-\rho_0)+\frac{1}{2}(|T-\id|^2, \mu_1-\mu_0)=0.
\]
In particular, we see that for $0<t_1<t_2<1$
\[
(\eta_{t_2}-\eta_{t_1}, \mu_1-\mu_0)+(\zeta_{t_2}-\zeta_{t_1}, \rho_1-\rho_0)=0.
\]
 Since $\partial_z f(\cdot,x)$ is strictly increasing for all $x$ in the range of $G(\cdot,x)$ due to the continuity of $G$,  and since  $\partial s(y)$ is increasing,
\[
(\eta_{t_2}-\eta_{t_1}, \mu_1-\mu_0)+(\zeta_{t_2}-\zeta_{t_1}, \rho_1-\rho_0)>0
\]
if $\mu_1\neq \mu_0$ almost everywhere, yielding a contradiction.  Thus we can conclude.

\end{proof}

In general, we do not expect that there is a unique pressure that maximizes the dual problem.  Luckily, the following Lemma guarantees that the set of maximizers  $\argmax_{p\in X^*,\, p^{c\bar{c}}=p}  J^*(p,\rho_n)$, has a minimal element.  Thus, we can always make a consistent choice of pressure in the scheme by setting $p^{n+1}$ to be the smallest maximizer.  Note that the proof of the lemma is parallel to the corresponding result, Lemma 4.4, of \cite{JKT}.

\begin{lemma} [Existence of minimal pressure]   \label{pressure_special} There exists  a $c$-concave function
\[
p^*\in \argmax_{p\in X^*,\, p^{c\bar{c}}=p}  J^*(p,\bar{\rho})
\]
 such that $p^*\leq \tilde{p}$ for any
$ \tilde{p}\in\argmax_{p\in X^*,\, p^{c\bar{c}}=p}  J^*(p,\bar{\rho}).$
\end{lemma}

\subsection{Uniform bounds}

Next we proceed to establish uniform bounds on both density and pressure variables. We begin by showing that the pressure stays bounded from above when the growth function $G$ is strictly decreasing.

\begin{lemma}\label{lem:linf_bound}
Let $b(x)$, $b_0$ and $b_1$ be as given in \ref{assumption: death when large pressure}, and in addition suppose that $G(\cdot,x) $ is strictly decreasing for all $x\in\Omega$. Let  $\hat{M}=\sup \partial s^*(b_1)$,  $\bar{M}=\inf \partial s^*(b_0)$.
For $p_0$ chosen from \eqref{good_data_0}, define for $n=0,1,\cdots$,
$$
\hat{M}_n:=\esssup_{x\in\Omega} \rho^n(x), \quad \bar{M}_n:=\essinf_{x\in \Omega} \rho^n(x), \hbox{ and } v_n:=\inf \partial s(\hat{M}_n).
$$ Then the following holds:
\begin{itemize}
\item[(a)] $p_{n+1}\leq P_{n+1}:= \max(v_n, b_1)$, and $P_{n+1}$ decreases with respect to $n$.
\item[(b)] $\max(\hat{M}, \hat{M}_n)$ is non-increasing and $\min(\bar{M}, \bar{M}_n)$ is non-decreasing with respect to $n$.
\end{itemize}

In particular, if $\rho_0$ satisfies \eqref{good_data_0}, then so is $\rho^n$, and the bound is independent of $n$ and $\tau$.
\end{lemma}

\begin{proof}

First let us define
$$
 U := \{x: p_{n+1}(x)> \max(v_n,b_1)\}. 
$$
From \eqref{eq:pd_relation_energy} we have that up to a measure zero set, $\{x\in\Omega: \rho^{n+1}(x)>\max(\hat{M}_n, \hat{M})\}\subset U$, and thus it is enough to show that $|U|=0$ to conclude both $(a)$ and the first claim in $(b)$.
Note that $p_{n+1}$ is continuous, so $|U| = 0$ implies $U = \varnothing$.

\medskip

 If $x\in T^{-1}_{p_{n+1}}(U)$, $p^c_{n+1}(x) \geq p_{n+1}(T_{p_{n+1}}(x))> b_1$. Moreover from Remark~\ref{dual2}  it follows that $p_{n+1}(x)\geq p_{n+1}(T_{p_{n+1}}(x))$. Since $G(\cdot,x)$ is strictly decreasing, we have

 \begin{equation}\label{sign}
 T^{-1}_{p_{n+1}}(U)\subset U \hbox{ and } G(p^c_{n+1}(x),x) < 0 \hbox{ in }T^{-1}_{p_{n+1}}(U),
\end{equation}

\medskip

Now let $\phi$ be the characteristic function of $U$. If $|U|\neq 0$, we can use \eqref{eq:pd_relation_pushforward} and \eqref{sign} to conclude that
\[
\begin{split}
\int_{U} \rho^{n+1}(x)\, dx= &\;\int_{\Omega} \rho^{n}(x)\Big(1+\tau G(p^c_{n+1}(x),x)\Big)\phi\big(T_{p_{n+1}}(x)\big)\,dx \\
< &\;\int_{\Omega} \rho^n(x)\phi(T_{p_{n+1}}(x))\,dx\leq \int_{U} \rho^n(x)\, dx.
\end{split}
\]
Lastly, note that  from the definition of $v_n$ and the dual relation we have $\rho^{n+1} \geq \hat{M}_n \geq \rho^n$ a.e. \;in $U$. It follows that $|U|=0$.

\medskip

Next let $V=\{x\in\Omega: \rho^{n+1}(x)<\min(\bar{M}_n, \bar{M})\}$.  It is immediate from \eqref{eq:pd_relation_growth} that $\mu^{n+1}(x) \geq 0$ in $V$. Moreover, since $b_0> p_{n+1}(x) \geq p_{n+1}(T_{p_{n+1}}(x))$ in $V$, it follows that $V\subset T_{p_{n+1}}^{-1}(V)$. Thus, 
\[
\int_{V} \rho^{n+1}(x)\,dx\geq \int_{T^{-1}_{p_{n+1}}(V)}\rho^n(x)+\tau\mu^{n+1}(x)\, dx \geq \min(\bar{M},\bar{M}_n)|V|,
\]
which yields $|V|=0$.
\end{proof}

 We can now extend the pressure bound to general $G$ by working with the minimal pressure from Lemma \ref{pressure_special}.

\begin{corollary}\label{cor:linfbound}
If we choose $p^{n+1}$ to be the minimal pressure, the statements of Lemma~\ref{lem:linf_bound} hold without the strict monotonicity assumption on $G$.
\end{corollary}

\begin{proof}
It is enough to show that $p_{n+1} \leq P_{n+1}$, since the rest of the proof does not use the strict monotonicity of $G$. Let us approximate $G$ by strictly decreasing functions $G_{\delta}(p,x):= G(p,x) + \delta (e^{-p} -1)$, and let $J^*_\delta$ be the corresponding dual energy and let $p_{n+1,\delta}$ be the minimal element of $\argmax_{p\in X^*,\, p^{c\bar{c}}=p} J^*_{\delta}(p,\rho^n)$ chosen by Lemma~\ref{pressure_special}. Then Lemma~\ref{lem:linf_bound} applies to $p_{n+1,\delta}$ to yield $p_{n+1,\delta} \leq P_{n+1}$. Due to the uniform Lipschitz continuity,  $p_{n+1,\delta}$ uniformly converges to an element  $p^*\in S_n=\argmax_{p\in X^*,\, p^{c\bar{c}}=p} J^*(p,\rho^n)$ along a subsequence as $\delta\to 0$. Since $p_{n+1}$ is the minimal element of $S_n$, we can conclude that $p_{n+1} \leq p^* \leq P_{n+1}$.
\end{proof}

Lastly we discuss BV estimates for the density variable. The following lemma will yield exponential growth of the BV norm over time (see Corollary~\ref{BV2}).

\begin{lemma}\label{BV}
Let $\Omega$ be  convex and bounded and let $P_n$ be given as in Lemma~\ref{lem:linf_bound}. If $g_0:=\norm{G}_{W^{1,\infty}\big([0,P_1]\times\R^d\big)}<\infty$ and $\rho^{n}\in BV(\Omega)$, then for $\tau<1/B$ we have $\rho^{n+1}\in BV(\Omega)$ with the bound
\begin{equation*}
\int_{\Omega} |\nabla \rho^{n+1}|dx \leq (1+\tau B)\int_{\Omega} |\nabla \rho^n(x)| dx+ \tau \left( \frac{1}{2(1-\tau B)}\int_{\Omega} |\nabla p_{n+1}|^2 \rho^{n+1} dx + \Big(\frac{g_0^2}{2}+g_0\Big)\rho^n(\Omega)\right).
\end{equation*}
\end{lemma}
\begin{proof}
From \eqref{eq:mms} we have $\rho^{n+1}=\argmin_{\rho\in L^1(\Omega)} W_2(\rho, \rho^n+\tau\mu^{n+1}) + E(\rho)$. Thanks to the useful estimate derived in \cite{bv_ot}, it follows that
\[
\int_{\Omega} |\nabla \rho^{n+1}|\leq \int_{\Omega} |\nabla (\rho^n+\tau\mu^{n+1})|.
\]
Using \eqref{eq:pd_relation_growth} we obtain
\[
\int_{\Omega} |\nabla \rho^{n+1}|\leq \int_{\Omega} |\nabla \rho^n(x)(1+\tau G(p^c_{n+1},x))|+\tau\rho^n(x)|\partial_zG( p^c_{n+1},x)\nabla p^c_{n+1}|+\tau\rho^n(x)|\partial_x G(p^c_{n+1}(x),x)|
\]
For the first term we have
\[
\int_{\Omega} |\nabla \rho^n(x)(1+\tau G(p^c_{n+1},x))| \leq (1+\tau B)\int_{\Omega}|\nabla \rho^n|.
\]
We estimate the second term with the Cauchy-Schwarz inequality,
\[
\tau \int_{\Omega}\rho^n|\partial_zG(p^c_{n+1},x)\nabla p^c_{n+1}| \leq  \frac{\tau}{2}\left(\int_{\Omega} \rho^n |\partial_zG(p^c_{n+1},x)|^2 dx + \int_{\Omega}\rho^n|\nabla p^c_{n+1}|^2 dx\right).
\]
From \eqref{eq:forward} and \eqref{eq:pd_relation_pushforward}, for $\tau < 1/B$,
\[
\begin{split}
\int_{\Omega}\rho^n|\nabla p^c_{n+1}|^2\leq &\; (1-\tau B)^{-1}\int_{\Omega}(\rho^n + \tau\mu^{n+1})|\nabla p^c_{n+1}|^2\\
= &\; (1-\tau B)^{-1} \tau^{-1}W_2^2(\rho^{n+1}, \rho^n+\tau\mu^{n+1})= (1-\tau B)^{-1}\int_{\Omega}\rho^{n+1}|\nabla p_{n+1}|^2.
\end{split}
\]
Now the stated estimate follows from the definition of $g_0$.
\end{proof}

\subsection{Monotonicity properties}

Here we study monotonicity properties of the density variable. We first show that the support of the density variable only expands over time.   Note that throughout  this subsection we are assumming that $\rho^n$ satisfies \eqref{good_data}.
\begin{lemma}\label{lem:expansion}
Up to a set of measure zero, $\{x\in\Omega: \rho^n(x)>0\} \subset \{x\in\Omega: \rho^{n+1}(x)>0\}$. 
\end{lemma}
\begin{proof}
If $D=\{x\in\Omega: \rho^n(x)>0, \, \rho^{n+1}(x)=0\}$, then $p_{n+1}(x)\leq 0$ for almost all $x\in D$.   Thanks to assumption \ref{assumption: growth at zero pressure}, it then follows that
\[
 \mu^{n+1}(x)=\rho^n(x)G(p^c_{n+1}(x),x)\geq \rho^n(x)G(0,x)>0\quad \textrm{a.e.}\;x\in D.
 \]
Choose some point $x\in D$ and note that if $T_{p_{n+1}}(x)\neq x$, then we must have $p_{n+1}(T_{p_{n+1}}(x))< p_{n+1}(x)$.   Therefore, for almost every $x\in D$, it follows that $\rho^{n+1}(T_{p_{n+1}}(x))\leq \rho^{n+1}(x)$.  Now we can compute
\[
\int_{D} \rho^n(x)+\tau\mu^{n+1}(x)\, dx=\int_{T_{p_{n+1}}(D)} \rho^{n+1}(y)\, dy\leq 0.
\]
This is only possible if $D$ has measure zero.
\end{proof}

When we consider the singular energy $E_{\infty}(\rho)$, we can strengthen the previous proposition to a powerful monotonicity statement. In addition, we show that if the density starts out as a characteristic function it remains a characteristic function for all time.  Recall that in this case
$$
X = \{\rho\in L^1(\Omega): E(\rho) <\infty\} = \{\rho\in L^1(\Omega): 0\leq \rho \leq 1 \mbox{ a.e.}\}.
$$
\begin{prop}\label{density_monotone}
Let $s=s_{\infty}$ with $\rho^n\in X$. Then $\rho^{n+1} \geq \rho^n$ almost everywhere. Moreover if $\rho^n(x)\in \{0,1\}$ a.e.\;in $\Omega$, then we have $\rho^{n+1}\in \{0,1\}$ a.e.\;in $\Omega$.
\end{prop}
\begin{proof}
From Lemma \ref{lem:expansion}, it follows that $D=\{x\in\Omega: \rho^n(x)>0, \rho^{n+1}(x)=0\}$ has measure zero.
Let us now consider the set $E=\{x\in\Omega: 0<\rho^{n+1}(x)<\rho^n(x)\}$. This in particular means that $\rho^{n+1}(x)\in (0,1)$ on $E$, and thus from the relation $\rho^{n+1} \in \partial s_{\infty}^*(p_{n+1})$ it follows that $p_{n+1} = 0$ almost everywhere on $E$. $p_{n+1}$ is Lipschitz, so we also have $\nabla p_{n+1}=0$ almost everywhere on $E$ (see for instance Theorem 4.4 in \cite{evans_gariepy}).  $T^{-1}_{p_{n+1}}(x)=x+\tau\nabla p_{n+1}(x)$ is the optimal map from $\rho^{n+1}(x)$ to $\rho^n(x)+\tau\mu^{n+1}(x)$.  Therefore, for almost every $x\in E$, we have $T_{p_{n+1}}(x)=T^{-1}_{p_{n+1}}(x)=x$.  Since $p^c_{n+1}(x)
\leq p_{n+1}(x) $,  we can conclude that $\mu^{n+1}(x)\geq \rho^nG(0,x)$ for almost all $x\in E$.  Now we can compute
\[
\int_{E} \rho^{n+1}(x) dx= \int_{E} \rho^n(x)+\tau\mu^{n+1}(x)\, dx \geq\int_{E} \rho^n(x)\big(1+\tau G(0,x)\big)\, dx\geq \int_{E} \rho^n(x)\, dx.
\]
The above is only possible if $E$ has measure zero.

\medskip

It remains to show that if $\rho^n\in \{0,1\}$ almost everywhere, then  $\rho^{n+1}\in \{0,1\}$ almost everywhere.  Let $A=\{x\in\Omega: \rho^{n+1}(x)\in (0,1)\}$. Arguing as before, we conclude that $p_{n+1}(x)=0$ and $T_{p_{n+1}}(x)= T^{-1}_{p_{n+1}}(x)=x$ for almost all $x\in A$.  Furthermore, we have $\rho^n(x)\leq \rho^{n+1}(x)<1$ for almost all $x\in A$, thus, $\rho^n(x)\in \{0,1\}$ implies $\rho^n(x)=0$ for almost all $x\in A$.  Now we see that
\[
\int_A \rho^{n+1}(x)\, dx=\int_A \rho^n(x)+\tau\mu^{n+1}(x)\, dx\geq \int_A \rho^n(x)\big(1+\tau G(0,x)\big)\, dx=0.
\]
Thus, $A$ must have zero measure.
\end{proof}

\section{Comparison Principles}
\label{sec: comparison}
In this section we establish a comparison principle for both the density and pressure variables.  One main ingredient in the proof is \eqref{observation}, a property of optimal maps which played a central role in the $L^1$-contraction result in \cite{JKT}. In the case of $s=s_m$, the comparison principle was shown in \cite{AKY14} where \eqref{observation} was implicitly used.

We begin by establishing a comparison principle for the pressure variables.  Note that we only compare the positive parts of the pressure.  This is because we can only guarantee the comparison property on regions where the density variables do not vanish.

\begin{prop}\label{comparison0} For $i\in \{0,1\}$ let $\rho_i\in X$  and $G_i(p,x)$ be growth functions satisfying assumptions \ref{assumption: growth at zero pressure}-\ref{assumption: bound} with $G_0$ strictly decreasing in $p$. Define
\[
J_i^*(p):= \int_{\Omega} \rho_i(x)\big( p^c(x)+\tau \bar{G}_i(p^c(x),x)\big)-s^*(p(x))\, dx,
\]
and suppose that the following properties hold:
\begin{enumerate}
\item  $0\leq\rho_0 \leq \rho_1$ a.e.\;in $\Omega$ and they satisfy \eqref{good_data};
\item $ G_0(z,x)\leq G_1(z,x)$ for all $(z,x)\in\RR\times\Omega$.
\end{enumerate}
If
\[
p_i\in \argmax_{\{p\in X^*, \; p=p^{c\bar{c}}\}} J_i^*(p),
\]
then $(p_0)_+ \leq (p_1)_+$.

\end{prop}

\begin{proof}
Lemma~\ref{gangbo} and the optimality of $p_i$'s imply that there exists $\eta_i\in \partial s^*(p_i)$ such that
\begin{equation}\label{eq:dual_comparison_optimality}
\int_{\Omega} \rho_i(x)\big(1+\tau G_i(p_i^c(x),x)\big)\phi(T_{p_i}(x))-\eta_i(x)\phi(x)\, dx=0,
\end{equation}
for any bounded function $\phi$. Since $\eta_i \geq 0$, we have
\begin{equation}\label{eqn: positivity of density to be transported}
\rho_i(x)\big(1+\tau G_i(p_i^c(x),x)\big)\geq 0\quad\mbox{a.e.\;in }\Omega.
\end{equation}

Let $U=\{x\in\Omega: p_0(x)>p_1(x)\}$ and choose $\phi$ to be the characteristic function of $U$.
If we subtract \eqref{eq:dual_comparison_optimality} with $i=0$ from \eqref{eq:dual_comparison_optimality} with $i=1$ and rearrange, we see that
\[
\begin{split}
&\;\int_{\Omega}\rho_1(x)\Big(1+\tau G_1\big(p_1^c(x),x\big)\Big)\phi\big(T_{p_1}(x)\big)-\rho_0(x)\Big(1+\tau G_0\big(p_0^c(x),x\big)\Big)\phi\big(T_{p_0}(x)\big) \, dx\\
&\;=\int_{\Omega} \big(\eta_1(x)-\eta_0(x)\big)\phi(x)\, dx.
\end{split}
\]
Since $\partial s^*$ is increasing, it is clear that $\eta_0(x)\geq\eta_1(x)$ on $U$.  Hence,
\begin{equation}\label{eq:comparison_contradiction}
\int_{\Omega}\rho_1(x)\Big(1+\tau G_1\big(p_1^c(x),x\big)\Big)\phi\big(T_{p_1}(x)\big) dx \leq \int_{\Omega}\rho_0(x)\Big(1+\tau G_0\big(p_0^c(x),x\big)\Big)\phi\big(T_{p_0}(x)\big) \, dx.
\end{equation}

By Lemma 4.1 in \cite{JKT}, we know that
 \begin{equation}\label{observation}
\hbox{ if }  T_{p_0}(x)\in U, \hbox{ then } T_{p_1}(x)\in U,
\end{equation}
which gives $\phi(T_{p_0}(x))\leq \phi(T_{p_1}(x))$.
Moreover, if $T_{p_0}(x)\in U$, then
\[
p^c_1(x)=p_1(T_{p_1}(x))+\frac{1}{2\tau}|T_{p_1}(x)-x|^2\leq p_1(T_{p_0}(x)) + \frac{1}{2\tau}|T_{p_0}(x)-x|^2 < p_0^c(x),
\]
where the second inequality uses the definition of $p_1^c$ and the third uses the fact that $p_1<p_0$ at $T_{p_0}(x)$.
Due to the strict monotonicity of $G_0$, we obtain that
\begin{equation}
\label{eqn: strict order of growth term}
G_0\big(p_0^c(x),x\big)<G_1\big(p_1^c(x),x\big) \quad \textup{if} \; T_{p_0}(x)\in U.
\end{equation}
So we derive by \eqref{eqn: positivity of density to be transported} that
\[
\begin{split}
&\;\int_{\Omega}\rho_1(x)\Big(1+\tau G_1\big(p_1^c(x),x\big)\Big)\phi\big(T_{p_1}(x)\big)-\rho_0(x)\Big(1+\tau G_0\big(p_0^c(x),x\big)\Big)\phi\big(T_{p_0}(x)\big) \, dx\\
\geq&\;
\int_{\Omega}\Big(\rho_1(x)\big(1+\tau G_1\big(p_1^c(x),x\big)\big)-\rho_0(x)\big(1+\tau G_0\big(p_0^c(x),x\big)\big)\Big)\phi\big(T_{p_0}(x)\big)\, dx\\
\geq&\;
\tau\int_{\Omega}\Big( G_1\big(p_1^c(x),x\big)-  G_0\big(p_0^c(x),x\big)\Big)\rho_0(x)\phi\big(T_{p_0}(x)\big)\, dx\geq 0.
\end{split}
\]
In the last line, we used \eqref{eqn: positivity of density to be transported} and the assumption that $\rho_0\leq \rho_1$ a.e.\;in $\Omega$.
This together with \eqref{eq:comparison_contradiction} and \eqref{eqn: strict order of growth term} implies that $\rho_0(x)\phi(T_{p_0}(x))=0$ a.e.\;in $\Omega$.
Let $\rho_i^*:=T_{p_i\#}(\rho_i(1+\tau G_i(p_i^c(x),x)))$ are the optimal densities in the corresponding primal problem.the proof. Then
\[
\int_{U} \rho_0^*(x)\, dx=\int_{\Omega} \rho_0(x)(1+\tau G_0(p_0^c(x),x))\phi(T_{p_0}(x))\, dx=0.
\]
    Recall that Proposition \ref{primal_dual} shows $\rho_i^*\in \partial s^*(p_i)$ a.e..
Since $\partial s^*$ is increasing, we have $\rho_1^*\leq \rho_0^*$ a.e.\;on $U$, which implies
\[
\int_{U} \rho_1^*(x)\, dx\leq \int_{U} \rho_0^*(x)\, dx=0.
\]
This allows us to conclude that $p_0(x)\leq p_1(x)$ for  $(\rho_1^*+\rho_0^*)$ a.e. $x$.  From the dual relation $p_i\in \partial s(\rho^*_i)$ it follows that if $p_i>0$ then $\rho_i^*>0$.  Taking the positive part of the pressures the result now follows.
\end{proof}

\begin{corollary}\label{comparison00}
The statement of Proposition~\ref{comparison0} holds without the strict monotonicity assumption on $G_0$.
\end{corollary}
\begin{proof}
We proceed using an argument similar to the proof of Corollary~\ref{cor:linfbound}. We approximate $G$ from below by setting $G_{\delta}(p,x):= G_0(p,x) +\delta (e^{-p}-1)$.  Let $J^*_\delta$ be the corresponding dual energy and let $p_\delta$ be the minimal element of $\argmax_{p\in X^*,\, p^{c\bar{c}}=p} J^*_{\delta}(p,\rho_0)$ chosen by Lemma~\ref{pressure_special}. Then Proposition~\ref{comparison0} applies to $G_{\delta}$ to yield $p_\delta \leq p_1$. Due to the uniform Lipschitz continuity,  $p_\delta$ uniformly converges to an element  $p^*\in S_0=\argmax_{p\in X^*,\, p^{c\bar{c}}=p} J^*(p,\rho_0)$ along a subsequence as $\delta\to 0$. Since $p_0$ is the minimal element of $S_0$, we can conclude that $(p_0)_+ \leq (p^*)_+ \leq (p_1)_+$.
\end{proof}

Next we extend the comparison principle to the density variable.

\begin{prop}\label{comparison}
For $i\in \{0,1\}$ let $J_i:X\times \textup{AC}(\rho_i)$ be the functional
\[
J_i(\rho,\mu)=E(\rho)+\frac{1}{2\tau}W_2^2(\rho, \rho_i+\tau\mu)+\int_{\Omega} \tau\rho_i f_i\left(\frac{\mu(x)}{\rho_i(x)},x\right)\, dx
\]
and suppose that the following properties hold
\begin{enumerate}
\item $0\leq\rho_0 \leq \rho_1$ a.e.\;and they satisfy \eqref{good_data};
\item $ \partial_z f_1(z,x)\leq\partial_z f_0(z,x)\; \textrm{for all} \;\; (z,x)\in \RR\times \Omega.$
\end{enumerate}
If
\[
(\rho_i^*,\mu_i^*)=\argmin_{(\rho,\mu)\in X\times \textup{AC}(\rho_i)} J_i(\rho,\mu)
\]
then $\rho_0^*\leq \rho_1^*$ a.e.\;in $\Omega$.
\end{prop}
\begin{proof}
Note that our primal problems correspond to the dual problems $J_i^*:X^*\to \RR$ where
\[
J_i^*(p)= \int_{\Omega} \rho_i(x)\big( p^c(x)+\tau \bar{G}_i(p^c(x),x)\big)-s^*(p(x))\, dx.
\]
Since $G_i(-z,x)=\partial_z f^*(z,x)$, $G_0(z,x)\leq G_1(z,x)$ for all $(z,x)\in \RR\times \Omega$.  Due to Proposition~\ref{primal_dual} and Corollary~\ref{comparison00}, there exist $p_0, p_1\in \argmax_{\{p\in X^*, \; p=p^{c\bar{c}}\}}J_i^*(p)$ such that $p_0\leq p_1$ on $\spt\rho_0^*\cup\spt\rho_1^*$ and
\[
\rho_i^*\in \partial s^*(p_i), \quad T_{p_i\#}(\rho_i+\tau\mu^*_i)=\rho_i^*, \quad \mu_i^*(x)=\rho_i(x) G(p_i^c(x),x).
\]
Now let $E=\{y\in\Omega:\rho_1^*(y)< \rho_0^*(y)\}$.  Since $E\subset\spt\rho_0^*$ and $\partial s$ is increasing, we must have $p_0=p_1$ almost everywhere on $E$.  Therefore, $\nabla p_0=\nabla p_1$ a.e in  $E$ (see e.g.\;Section 4.2 of \cite{evans_gariepy}), which gives $E=T_{p_1}^{-1}(E)=T_{p_0}^{-1}(E)$ (up to sets of measure zero) and $\mu_0^*(x)\leq \mu_1^*(x)$ for almost all $x\in E$.   Therefore,
\[
\int_E \rho_1^*(y)-\rho_0^*(y)\, dy=\int_{E} \rho_1(x)-\rho_0(x) + \tau(\mu_1^*(x)-\mu_0^*(x))\geq 0.
\]
Thus, $\rho_0^*(x)\leq \rho_1^*(x)$ for almost all $x\in\Omega$.
\end{proof}

Iterating Proposition~\ref{comparison0} and Proposition~\ref{comparison}, we have the following.

\begin{corollary}\label{comp}
Let $\rho_0, \rho_1$ and $G_0,G_1$ as given in Proposition~\ref{comparison0}.  Let us denote $\{\rho^n_i\}_n$ and $\{p_{n,i}\}_n$ as the sequence of solutions generated respectively by \eqref{eq:mms}  and \eqref{eq:dual_scheme} with initial data $\rho_i$. Then for $i=1,...,N$ we have
$$
\rho^n_0 \leq \rho^n_1\hbox{ and } (p_{n,0})_+ \leq (p_{n,1})_+,
$$
as long as $\{\rho^n_i\}_{n=1}^N$ satisfy \eqref{good_data}.
\end{corollary}

Combining this with Proposition~\ref{density_monotone},  we obtain the following corollary.

\begin{corollary}\label{pressure_monotone}
Let $s=s_{\infty}$ and let $\rho_0 \in X$. Then both $\rho^n$ and $(p_n)_+$ increase a.e.\;with respect to $n$ as long as $\rho^1,\cdots,\rho^n$ satisfies \eqref{good_data}.
\end{corollary}

\section{Finite-speed Propagation Property}
\label{sec: finite speed propagation}
Based on the comparison principle, we will show that the support of the densities propagate with finite speed, see Proposition~\ref{finite_prop}. Since the support only expands in time (Lemma~\ref{lem:expansion}), it is enough to obtain an upper bound on the expansion rate.
This result in particular allows us to obtain our unique discrete solutions independently on the choice of domain $\Omega$, as long as it is sufficiently large. See Corollary \ref{cor:domain}.

\medskip

The main step in this section is the following construction of radial barriers.
\begin{prop}\label{lem: basic barrier functions}
Fix $\tau>0$ and a finite number $\rho^+\in \partial s^*([0,+\infty))$.
There exist universal positive constants $R_*$ and $c_*$, whose upper bounds depend on $G$, $\partial s^*$ and $\rho^+$, and a family of densities $\{\rho_R\}_{R\geq R_*}$, which additionally depends on $\tau$, such that
\begin{enumerate}
  \item $\rho_R$ is radially symmetric and supported on $\overline{B_R}$, with $\rho_R\geq \rho^+$ on $B_{R-R_*}$;
  \item The new optimal density $\rho$  obtained by \eqref{eq:mms} with $\rho^n$ replaced by $\rho_R$ satisfies $\rho\leq \rho_{R+c_*\tau}$ a.e., provided that $B_{R+c_*\tau}\subset \Omega$.
\end{enumerate}
\end{prop}

Note that given $\rho_0\in X$ satisfying \eqref{good_data_0}, we must have $\|\rho_0\|_{L^\infty(\Omega)}$ be finite and lie in $\partial s^*([0,+\infty))$. Thanks to Proposition \ref{comparison}, it immediately implies

\begin{corollary}[Densities propagate with finite speed]\label{finite_prop}
Given $\rho_0\in X$ satisfying \eqref{good_data_0}, let $R_*$, $c_*$ and $\{\rho_R\}$ be as in Lemma \ref{lem: basic barrier functions} with $\rho^+:=\|\rho_0\|_{L^\infty(\Omega)}$. For fixed $\tau>0$, let $\rho^n = \rho^{n,\tau}$ be as given in \eqref{eq:mms} starting from $\rho_0$. If $\rho_0\leq \rho_{R_0}$ a.e.\;for some $R_0\geq R_*$, then we have
$$
\rho^n \leq \rho_{R_n} \hbox{ where } R_n:= R_0 + nc_*\tau, \hbox{ as long as }  B_{R_n}\subset \Omega.
$$

 In particular, suppose that $\spt\rho_0\subset \overline{B_R(0)}$ for some $R>0$. Then $\spt\rho^n\subset \overline{B_{R_n}(0)}$ with $R_n : =R+R_*+nc_*\tau$ provided that the latter ball is contained in $\Omega$.

\end{corollary}

Due to the uniqueness of $\rho^n$ (Proposition~\ref{primal_dual}), the following holds.
\begin{corollary}\label{cor:domain}
Suppose that $\rho_0\in X$ satisfies \eqref{good_data_0} with $\spt{\rho_0} \subset B_R(0)$. Then for $n\tau \leq T$ the sequence $\{\rho^n\}$ is independent of the choice of domain $\Omega$, as long as $\Omega$ contains $B_{R+R_* + c_*T}(0)$.
\end{corollary}

In the rest of this section we work toward the proof of Proposition \ref{lem: basic barrier functions}.

\medskip

We start by taking a smooth decreasing function $\tilde{G} = \tilde{G}(p)$, such that
\begin{enumerate}
  \item $\sup_{x\in \Omega}G(z,x)\leq \tilde{G}(z)$ for all $z\in \R$.
  \item $\tilde{G}(0)<+\infty$ and satisfies $\tilde{G}(z_M) = 0$, $s^*$ is differentiable at $z_M$ and $(s^*)'(z_M) \geq \rho^+$ for some $z_M>0$.
\end{enumerate}
Indeed, thanks to \ref{assumption: death when large pressure} and \ref{assumption: bound}, such $\tilde{G}$ exists and only depends on $G$ and $\rho^+$.

Let $\tilde{f} = \tilde{f}(z)$ be defined by $\tilde{G}(z)$ by \eqref{f_def} as $f$ is determined by $G$, and consider the modified dual problem with an $x$-independent growth term,
\begin{equation}
\sup_{q^{\bar{c}}\in X^*,\,q^{\bar{c}c}=q} \int_\Omega \rho_0(x)\big(q(x)+\tau \bar{\tilde{G}}(q(x))\big)\,dx-s^*(q^{\bar{c}}),
\label{eqn: dual problem repeated}
\end{equation}
where $\bar{\tilde{G}}(z) := z\tilde{G}(z)+\tilde{f}(\tilde{G}(z))$ is an anti-derivative of $\tilde{G}$.
By Proposition \ref{primal_dual}, \eqref{eqn: dual problem repeated} admits a maximizer, since we may introduce $p = q^{\bar{c}}$ and then $p\in X^*$ and $p$ is $c$-concave satisfying that $q = p^c$.
With abuse of notations, we still denote the maximizer by $q$.

Instead of treating $q$ to be determined by $\rho_0$, we shall first propose the optimal $q$ and then derive the corresponding $\rho_0$. 
For this purpose, let us assume that
\begin{enumerate}[label=(\roman*)]
  \item\label{barrier rho: symmetry} $\rho_0$ is radially symmetric, supported on $\overline{B_R}\subset \Omega$;
  \item\label{barrier rho: positive} $\rho_0>0$ on $B_R$ and $\rho_0\ll \mathcal{L}^d$;

  \item\label{barrier q: regularity} The optimal $q$ is radially symmetric and $q\in C^2(\overline{B_R})$.
\end{enumerate}

By radial symmetry, we may write $\rho_0 = \rho_0(r)$ and $q=q(r)$.
By definition, $q^{\bar{c}}(r-\tau q'(r)) = q(r)-\frac{\tau}{2}|q'(r)|^2$.
Taking $q$-variation in \eqref{eqn: dual problem repeated}, we obtain the optimality condition for $q$

\begin{equation}
\rho_0(r)\cdot \frac{1+\tau \tilde{G}(q(r))}{(1-\tau q''(r))(1-\tau r^{-1} q'(r))^{d-1}} \in  \partial s^*\left(q(r) - \frac{\tau}{2}|q'(r)|^2\right)\quad a.e.\;r\in[0,R).
\label{eqn: relation between old density and pressure q radial case}
\end{equation}

Note that this is also implied by \eqref{eq:pd_relation_energy}-\eqref{eq:pd_relation_growth}.

\medskip

Let $Q = Q(w)$ solve the following ODE on $w\geq 0$,
\begin{equation}
-Q''(w) = \tilde{G}(Q(w))+w,\quad Q(0) = z_M,\quad Q'(0) = 0.
\label{eqn: def of Q}
\end{equation}
It is then straightforward to show the following:
\begin{lemma}\label{lem: property of auxiliary function Q}
\begin{enumerate}
  \item There exists a unique $w_0>0$ only depending on $\tilde{G}$, such that $Q(w_0)=0$.
  \item $Q$ is smooth and $Q'(w), Q''(w)\leq 0$ for $w\in [0,w_0]$;
    \item There exists a unique $w_1 \in [0,w_0]$, such that
$Q(w_1) =\frac{\tau}{2}|Q'(w_1)|^2$.
Moreover, $c_*: = |Q'(w_1)|$ is bounded by some universal constant that only depends on $\tilde{G}$.
\end{enumerate}
\end{lemma}
Let $R_* = w_1+1$.
For any $R\geq R_*$, we define
\begin{equation}
q_R(r) =
\begin{cases}
z_M& \mbox{ if }r\leq R-w_1,\\
Q(r-(R-w_1))& \mbox{ if }r\in(R-w_1,R],\\
-\infty&\mbox{ otherwise.}
\end{cases}
\label{eqn: def of q_R}
\end{equation}
Note that \eqref{eqn: def of Q} and Lemma \ref{lem: basic barrier functions} guarantees $q_R\in C^2(\overline{B_R})$.
We also define according to \eqref{eqn: relation between old density and pressure q radial case} that 
\begin{equation}
\rho_R(r)
\begin{cases}
= (s^*)'(z_M)& \mbox{ if }r\leq R-w_1,\\
\in  \frac{(1-\tau q_R''(r))(1-\tau r^{-1} q_R'(r))^{d-1}}{1+\tau \tilde{G}(q_R(r))}
\cdot \partial s^*\left(q_R(r) - \frac{\tau}{2}|q_R'(r)|^2\right) & \mbox{ if }r\in(R-w_1,R],\\
=0&\mbox{ otherwise.}
\end{cases}
\label{eqn: precise def of rho_R}
\end{equation}
We need to justify that $\rho_R$ is well-defined a.e.\;in $\R^d$, in particular in the annular region $\{y\in \R^d:~|y|\in[R-w_1,R]\}$.
Take arbitrary $R'\in (R-w_1,R]$, and define
$$S_{R'}: = \Big\{r\in [R',R]:~ s^*(\cdot)\mbox{ is not differentiable at }q_R(r)-\frac{\tau}{2}| q_R'(r)|^2\Big\}.
$$ It is enough to show that $S_{R'}$ has measure zero.
Note that $q_R(r)-\frac{\tau}{2}| q_R'(r)|^2$ is $C^1$ and strictly decreasing on $[R',R]$, satisfying that
$$
\frac{d}{dr}\left(q_R(r)-\frac{\tau}{2}| q_R'(r)|^2\right)\leq C(R')<0\quad \forall\, r\in[R',R].
$$
Then $S_{R'}$ must have measure zero by virtue of the area formula \cite{simon1983lectures} and the fact that $s^*(\cdot)$ is convex and thus a.e.\;differentiable.
Since $R'$ can be arbitrarily close to $R-w_1$, we conclude that $\rho_R$ is well-defined a.e.\;in $\R^d$.
Hence, it is valid to write
\begin{equation}
\rho_R(r)= \frac{(1-\tau q_R''(r))(1-\tau r^{-1} q_R'(r))^{d-1}}{1+\tau \tilde{G}(q_R(r))}
\cdot (s^*)'\left(q_R(r) - \frac{\tau}{2}|q_R'(r)|^2\right) \quad a.e.~r\in[0,R].
\label{eqn: simplified relation between old density and pressure q radial case}
\end{equation}
$\rho_R$ satisfies the assumptions \ref{barrier rho: symmetry}-\ref{barrier rho: positive}.

Define $\tilde{\rho}$ by
\begin{equation}
\tilde{\rho}(r+\tau|q_R'(r)|)
\begin{cases}
= (s^*)'(z_M)& \mbox{ if }r\leq R-w_1,\\
\in  \partial s^*\left(q_R(r) - \frac{\tau}{2}|q_R'(r)|^2\right) & \mbox{ if }r\in(R-w_1,R],\\
=0&\mbox{ otherwise.}
\end{cases}
\label{eqn: the optimal new density}
\end{equation}
By a similar argument as above, $\tilde{\rho}$ is well-defined a.e.\;in $\R^d$.
Let us note that if we consider the following modified dual problem,
\[
\sup_{p\in X^*,\, p^{c\bar{c}}=p}\int_{\Omega} \rho_R(x)\Big(p^c(x)+\tau \bar{\tilde{G}}\big(p^c(x)\big)\Big)\, dx -\int_{\Omega} s^*(p(x)) dx,
\]
which is equivalent to \eqref{eqn: dual problem repeated} via $c$- and $\bar{c}$-transforms, by Proposition \ref{primal_dual}, the optimal $p$ is uniquely given by $q_R^{\bar{c}}$ on $\spt \tilde{\rho}$, while $\tilde{\rho}$ defined in \eqref{eqn: the optimal new density} is exactly the optimal new density.

For $\tilde{\rho}$, we can additionally show that
\begin{lemma}\label{lem: new density is dominated by another barrier}
$\tilde{\rho}\leq \rho_{\tilde{R}}$ almost everywhere, where $\tilde{R} := R+\tau |q_R'(R)|$.
\begin{proof}
By Lemma \ref{lem: property of auxiliary function Q} and \eqref{eqn: def of q_R}, $q_{\tilde{R}}'(r),q_{\tilde{R}}''(r)\leq 0$ for  $r\in [0,\tilde{R}]$.
By \eqref{eqn: def of Q} and \eqref{eqn: simplified relation between old density and pressure q radial case}, \begin{equation*}
\begin{split}
\rho_{\tilde{R}}(r) 
\geq &\;\frac{1-\tau q_{\tilde{R}}''(r)}{1+\tau \tilde{G}(q_{\tilde{R}}(r))}\cdot (s^*)'\left(q_{\tilde{R}}(r) - \frac{\tau}{2}|q_{\tilde{R}}'(r)|^2\right)\\
\geq  &\; (s^*)'\left(q_{\tilde{R}}(r) - \frac{\tau}{2}|q_{\tilde{R}}'(r)|^2\right).
\end{split}
\end{equation*}
Since $\tilde{\rho}$ is defined by \eqref{eqn: the optimal new density} and $(s^*)'$ is non-decreasing, it suffices to show that for all $r\in [0,R]$,
\begin{equation*}
q_{\tilde{R}}(r+\tau|q_R'(r)|) - \frac{\tau}{2}|q_{\tilde{R}}'(r+\tau|q_R'(r)|)|^2 \geq q_R(r)-\frac{\tau}{2}|q_R'(r)|^2.
\end{equation*}
By Lemma \ref{lem: property of auxiliary function Q}, $q_{\tilde{R}}$ is a decreasing function while $|q_{\tilde{R}}'|$ and $|q_{R}'|$ are increasing.
Hence,
\begin{equation*}
\begin{split}
&\;q_{\tilde{R}}(r+\tau|q_R'(r)|) - \frac{\tau}{2}|q_{\tilde{R}}'(r+\tau|q_R'(r)|)|^2\\
\geq &\;q_{\tilde{R}}(r+\tau|q_R'(R)|) - \frac{\tau}{2}|q_{\tilde{R}}'(r+\tau|q_R'(R)|)|^2\\
= &\; q_R(r)-\frac{\tau}{2}|q_R'(r)|^2.
\end{split}
\end{equation*}
The last step is derived from \eqref{eqn: def of q_R} and the definition of $\tilde{R}$.
\end{proof}
\end{lemma}

Now the proof of Proposition \ref{lem: basic barrier functions} is around the corner.

\begin{proof}[Proof of Proposition \ref{lem: basic barrier functions}]

Let $\rho$ be the new optimal density corresponding to $\rho_R$ obtained by the original discrete scheme \eqref{eq:mms}, while $\tilde{\rho}$ is obtained from the primal problem associated with \eqref{eqn: dual problem repeated}.

By the definition of $\tilde{G}$ and the monotonicity of $G(\cdot,x)$, it is not difficult to verify that $\partial_z \tilde{f}(z)\leq \partial_z f(z,x)$.
Thanks to Proposition \ref{comparison}, $\rho\leq \tilde{\rho}$ almost everywhere.
Hence, $\rho\leq \rho_{R+c_*\tau}$ follows from Lemma \ref{lem: new density is dominated by another barrier} and the fact that $|q_R'(R)|= c_*$.
\end{proof}

\medskip

\noindent{\it $\circ$ Finer construction of barriers}
\medskip

It is possible to construct a refined family of barriers $\{\rho_{R,A}\}$, which describes finer features of the density propagation.
Roughly speaking, $\rho_{R,A}$ has a radial, plateau-shaped profile with support $\overline{B_R}$, but the boundary transition of $\rho_{R,A}$ from 0 to its maximum value takes place in an annular region of width $O(A^{-1})$. With these barriers we obtain the following proposition, describing relaxation of the spreading speed for densities which with initially steep profiles near the boundary.
Since this result does not affect the rest of the paper, we postpone its proof as well as more detailed discussions in the Appendix \ref{appendix: improved barriers}.

\begin{prop}[Relaxation of the propagation speed]\label{prop_iteration}
Suppose $\rho_0\in L^{\infty}(\Omega)\cap X$ satisfies \eqref{good_data_0}. Fix $\tau\in (0,1]$.
Let $\{\rho^n\}$ be the sequence of densities obtained by the discrete scheme \eqref{eq:mms} starting from $\rho_0$.
Take $\rho^+ =\|\rho_0\|_{L^\infty(\Omega)}$ and let $R_*$ and $c_*$ be defined as in Lemma \ref{lem: basic barrier functions}.

With $A\geq 1$, let $\rho_{R,A}$ be defined in \eqref{eqn: precise def of rho_RA}.
Suppose $\rho_0\leq \rho_{R_0,A_0}$ for some $A_0\geq 1$ and suitable $R_0$. Then there is a sequence $\{R_n\}_{n = 0}^\infty$ with $R_n - R_{n-1}\to \tau c_*$, such that
$\spt\rho^{n} \subset \overline{B_{R_n}}$ for all $n \in \N$ which satisfies $B_{R_n}\subset \Omega$.
\end{prop}

\section{Equicontinuity for the Densities}
\label{sec: equicontinuity}

In the remainder of this paper, we will focus on showing that the interpolations $\rho^{\tau}, \mu^{\tau}$ and $p^{\tau}$ defined in \eqref{interpolation} converge to a solution of the tumor growth PDE when $\tau$ goes to zero.

 We shall suppose that the initial data $\rho_0\in X$, satisfies \eqref{good_data_0}. For such $\rho_0$, we will choose the domain $\Omega$ in the following way. For a given time horizon $T>0$, we choose $\Omega$ sufficiently large so that
$$
B_{R_0 + CT} \subset \Omega,
$$
with $R_0$ and $C$ only depending on $\rho_0$, $G$ and $s$, so that (i) the interpolations $(\rho^{\tau}, \mu^{\tau}, p^{\tau})$ given in \eqref{interpolation} stay compactly supported in $B_{R_0+CT}$ and (ii) $\rho^{\tau}$ satisfies \eqref{good_data} for $0\leq t\leq T$. This is possible due to Proposition~\ref{finite_prop}. We then extend the values of the discrete-time solutions to zero outside of $\Omega$. Due to Corollary~\ref{cor:domain}, any $\Omega$ satisfying the above requirements produces the same discrete-in-time solutions for $0\leq t\leq T$. In the below discussions, we will use the extended notion of discrete solutions in $\R^d\times[0,T]$.

\medskip

Now we can turn to the main focus of this section, establishing spatial and temporal equicontinuity estimates  for the densities.  These equicontinuity estimates will be used  to obtain strong convergence of densities (see Corollary~\ref{L^1_conv}) using an Arzeli-Ascola-type argument given in \cite{JKT}.  Although our arguments unfold along similar lines to  \cite{JKT}, we require substantially new arguments to handle the growth term and the lack of continuous differentiability for $s^*$.
\medskip

\subsection{The energy dissipation inequality, BV bounds, and equicontinuity in time}

We begin with the standard ``energy dissipation inequality",  which will allow us to obtain both BV and equicontinuity-in-time estimates for the densities.

\begin{lemma}\label{lem: energy dissipation inequality}
Given a time horizon $T>0$, let $\rho^\tau$, $\mu^\tau$ and $p^\tau$ be defined as in \eqref{interpolation}. Then for $T' := \big(\lfloor \frac{T}{\tau}\rfloor+1\big) \tau$ we have
\begin{equation*}
E(\rho^{\tau}(\cdot,T))+\frac{1}{2}\int_0^{T'} \int_{\R^d}|\nabla p^{\tau}|^2\rho^{\tau}\, dx\, dt \leq E(\rho_0)+\int_0^{T'} \int_{\R^d} p^{\tau}\mu^{\tau}\, dx\, dt.
\end{equation*}
\end{lemma}
\begin{proof}
By Lemma~\ref{dual_relation} applied to $(\rho^{n,\tau}, p_{n+1, \tau})$ and the inequality $\rho p \leq s(\rho) + s^*(p)$,
\[
E(\rho^{n+1,\tau})-E(\rho^{n,\tau})\leq \int_{\R^d} p_{n+1,\tau}(x)\big(\rho^{n+1,\tau}(x)-\rho^{n,\tau}(x))\,dx.
\]
Using the first formula in \eqref{eq:pd_relation_pushforward} we can rewrite the right hand side of the previous formula as
\[
\int_{\R^d} \rho^{n+1,\tau}\Big(p_{n+1,\tau}(x)-p_{n+1,\tau} \big(T^{-1}_{p_{n+1,\tau}}(x)\big)\Big)+\tau\mu^{n+1,\tau}(x)p_{n+1,\tau}(x)\, dx.
\]
The convexity of the map $y\mapsto p_{n+1,\tau}(y)+\frac{1}{2\tau}|y-x|^2$  gives us
\[
p_{n+1,\tau}\big(T^{-1}_{p_{n+1,\tau}}(x)\big)+\frac{1}{2\tau}\big|T^{-1}_{p_{n+1,\tau}}(x)-x\big|^2\geq p_{n+1,\tau}(x)+\big(\nabla p_{n+1,\tau}(x), T^{-1}_{p_{n+1,\tau}}(x)-x\big).
\]
From the second formula in \eqref{eq:pd_relation_pushforward} it follows that
\[
p_{n+1,\tau}(x)-p_{n+1,\tau}\big(T^{-1}_{p_{n+1,\tau}}(x)\big)\leq -\frac{\tau}{2}|\nabla p_{n+1,\tau}(x)|^2.
\]
Therefore
\[
E(\rho^{n+1,\tau})-E(\rho^{n,\tau})+\frac{\tau}{2}\int_{\R^d} \rho^{n+1,\tau}|\nabla p_{n+1,\tau}(x)|^2\,dx\leq \tau\int_{\R^d} \mu^{n+1,\tau}(x)p_{n+1,\tau}(x)\, dx.
\]
Summing over $n$ from $0$ to $\lfloor\frac{T}{\tau}\rfloor$ we conclude.
\end{proof}

For $\rho_0\in X$ satisfying \eqref{good_data_0}, $p^{\tau}$ is uniformly bounded by $P_1$ defined in Lemma~\ref{lem:linf_bound}.
By the assumption \ref{assumption: bound}, for any $t\in [0,T]$, $\rho^{\tau}(t,\R^d)\leq e^{B(t+\tau)}\rho_0(\R^d)$. Hence,
$$
\int_0^{T'}\int_{\R^d} p^{\tau}\mu^{\tau}\, dx\,dt \leq P_1B\left(\tau \rho_0(\R^d)+\int_0^{T'-\tau}\rho^{\tau}(\R^d,t)\,dt\right)\leq  P_1 e^{BT'}\rho_0(\R^d).
$$
Combining this with Lemma \ref{lem: energy dissipation inequality} and the fact $T'\leq T+\tau$, we find that
\begin{corollary}\label{edi}
Assume $\tau \leq 1/B$ and let $T'$ be defined as in Lemma \ref{lem: energy dissipation inequality}.
Then
\begin{equation}\label{discrete_edi}
\frac{1}{2}\int_0^{T'} \int_{\R^d}|\nabla p^{\tau}|^2\rho^{\tau}\, dx\, dt \leq E(\rho_0) +M\rho_0(\R^d),
\end{equation}
where $M$ depends on $T$, $G$, $s$ and $\norm{\rho_0}_{L^\infty(\R^d)}$.
\end{corollary}

Based on Lemma~\ref{BV} and \eqref{discrete_edi}, we can estimate the growth of the BV norm for the density.  The BV estimate will play a crucial role in establishing the spatial equicontinuity of the densities in the following subsection.

\begin{corollary}\label{BV2}
Assume $\tau \leq (2B)^{-1}$.
Suppose $\rho_0 \in BV$.
There exists a constant $M$ depending on $T$, $G$, $s$ and $\norm{\rho_0}_{L^{\infty}(\R^d)}$, such that for all $t\in [0,T]$,
$$
F(t):=\int_{\R^d} |\nabla \rho^{\tau}|(x,t) \,dx \leq e^{Bt}\big((1+\tau B)F(0) + 2E(\rho_0)+ M\rho_0(\R^d)\big).
$$
\end{corollary}

\begin{proof}
Let $T'$ be defined as in Lemma \ref{lem: energy dissipation inequality}.
Due to Lemma~\ref{BV} and \eqref{edi}, for any $t\in [0,T]$, 
\begin{equation*}
\begin{split}
\int_{\R^d} |\nabla\rho^\tau(x,t)|\,dx
\leq &\; B \int_0^t\int_{\R^d} |\nabla\rho^\tau(x,s)|\, dx\, ds+(1+\tau B)\int_{\R^d} |\nabla \rho_{0}|\,dx\\
&\;+ \int_0^{T'} \int_{\R^d}|\nabla p^{\tau} |^2\rho^{\tau} \, dx\, ds + M_2\left(\int_0^t\rho^{\tau}(\R^d,s)\, ds+\tau \rho_{0}(\R^d)\right)\\
\leq &\;B \int_0^t\int_{\R^d} |\nabla\rho^\tau(x,s)|\, dx\, ds+(1+\tau B)\int_{\R^d} |\nabla \rho_{0}|\,dx\\
&\;+ 2M_1 + M_2B^{-1}e^{B(t+\tau)}\rho_{0}(\R^d).
\end{split}
\end{equation*}
where 
$M_1:= E(\rho_0) +M\rho_0(\R^d)$ and $M_2 := \frac{1}{2}g_0^2+g_0$.  

Here we used the fact that $\rho^{\tau}(\R^d,t) \leq e^{B(t+\tau)} \rho_0(\R^d)$.
Then we conclude by the Gronwall's inequality.
\end{proof}

Next we obtain an equicontinuity-in-time estimate for the discrete densities.

\begin{lemma} \label{time_density}
Assume $\tau \leq 1/B$ and define $\rho^\tau(x,t) := \rho_0(x)$ for $t<0$.
Then for any $T>0$
$$
\int_0^{T} \Big\|\frac{\rho^{\tau}(\cdot,t)-\rho^{\tau}(\cdot,t-\tau)}{\tau}\Big\|_{H^{-1}(\R^d)}^2\leq M,
$$
where $M$ depends only on $T$, $G$, $s$ and $\rho_0$.
\end{lemma}
\begin{proof}
Let $\phi$ be a smooth function, and consider
\[
\int_{\R^d} \frac{\rho^{n+1,\tau}(x)-\rho^{n,\tau}(x)}{\tau} \phi(x)\, dx=\int_{\R^d} \frac{\phi(x)-\phi\big(x+\tau\nabla p_{n+1,\tau}(x)\big)}{\tau}\rho^{n+1,\tau}(x)+\mu^{n+1,\tau}(x)\phi(x)\, dx.
\]
Applying the fundamental theorem of calculus, the previous line is equal to
\[
\int_{\R^d}\int_0^1  \nabla \phi\big(x+\tau \theta\nabla p_{n+1,\tau}(x)\big)\cdot \nabla p_{n+1,\tau}(x)\rho^{n+1}(x)+\mu^{n+1,\tau}(x)\phi(x)\, d\theta\, dx.
\]
Applying Cauchy-Schwarz we then have the bound
\[
\norm{\nabla \phi}_{L^2(\tilde{\rho}^{n+1,\tau})}\norm{\nabla p_{n+1,\tau}}_{L^2(\rho^{n+1})}+\norm{\phi}_{L^2(\R^d)}B\norm{\rho^{n,\tau}}_{L^\infty(\R^d)}^{1/2}\rho^{n,\tau}(\R^d)^{1/2}.
\]
where
\[
\tilde{\rho}^{n+1,\tau}:=\int_0^1 \rho^{n+1,\tau}_{\theta}d\theta\quad \mbox{ and }\quad
\rho_{\theta}^{n+1,\tau}:= (\id +\tau\theta\nabla p_{n+1,\tau})_{\#}\rho^{n+1,\tau}.
\]
$L^p$-norms are displacement convex \cite{otam}, so
\[
\norm{\tilde{\rho}^{n+1,\tau}}_{L^{\infty}(\R^d)}\leq \max\Big(\norm{\rho^{n,\tau}+\tau\mu^{n+1,\tau}}_{L^{\infty}(\R^d)}, \norm{\rho^{n+1,\tau}}_{L^{\infty}(\R^d)}\Big).
\]
Hence,
\[
\begin{split}
&\;\Big\|\frac{\rho^{n+1,\tau}-\rho^{n,\tau}}{\tau}\Big\|_{H^{-1}(\R^d)}\\
\leq &\;\max\Big(\norm{\rho^{n,\tau}+\tau\mu^{n+1,\tau}}_{L^{\infty}(\R^d)}, \norm{\rho^{n+1,\tau}}_{L^{\infty}(\R^d)}\Big)^{1/2}\norm{\nabla p_{n+1,\tau}}_{L^2(\rho^{n+1})}\\
&\;+\Big(\norm{\rho^{n,\tau}}_{L^{\infty}(\R^d)}\rho^{n,\tau}(\R^d)\Big)^{1/2}B.
\end{split}
\]
Now we can conclude by Lemma~\ref{lem:linf_bound} and Corollary~\ref{edi}.
\end{proof}

\subsection{Equicontinuity in space}

Based on the comparison principles from the previous section and the BV estimate in Lemma~\ref{BV}, we  establish a spatial equicontinuity property for $\rho^\tau$.

\begin{prop}\label{equicontinuity}
Let $\rho_0$ and $\rho^\tau$ as given above. For any $y\in \R^d$ we have
\begin{equation*}
\lim_{\epsilon\to 0}\, \sup_{0<\tau\leq (2B)^{-1}}\, \int_0^T \int_{\R^d} |\rho^{\tau}(x+\epsilon y,t)-\rho^{\tau}(x,t)|\,dx\,dt=0.
\end{equation*}
\end{prop}

Observe that  when $\rho_0\in BV$, the result is immediate as $\{\rho^\tau\}_\tau$ has uniform BV bound on $[0,T]$ due to Corollary~\ref{BV2}.  In the general case, we will obtain equicontinuity by approximating the initial data $\rho_0$ with BV densities.  In order for this strategy to work, we will need to be able to extend this approximation to all times.   This is accomplished in the following lemma, which states that the $L^1$-difference of two discrete solutions can be controlled in terms of their $L^1$-difference at the initial time.

\begin{lemma}\label{lem: contraction}
Let  $\rho_0, \rho_1\in X$ satisfy \eqref{good_data_0}, and let $\rho^{\tau}_i$ be given by \eqref{interpolation} with initial data $\rho_i$ for $i=0,1$.
Then for all $0\leq t\leq T$,
\begin{equation}\label{contraction}
\int_0^t\int_{\R^d}|\rho^\tau_1(x,s)-\rho^\tau_0(x,s)|\, dx \,ds \leq \frac{1}{B} (e^{Bt}-1)(1+\tau B) \|\rho_1-\rho_0\|_{L^1(\R^d)},
\end{equation}
where $B$ is given in \ref{assumption: bound}.
\end{lemma}

\begin{proof}
First we suppose $\rho_0 \leq \rho_1$.
For $i\in \{0,1\}$, let $(\rho^n_i, \mu^n_i)$ and $p^n_i$ be generated respectively from \eqref{eq:mms} and \eqref{eq:dual_scheme} by the initial data $\rho_i$. By Proposition \ref{comparison}, for all $1\leq k \leq T/\tau$, we have $\rho_0^k\leq \rho_1^k$ almost everywhere.
Thus,
\[
  \begin{split}
   \int_{\R^d}|\rho^{n+1}_1(x)-\rho^{n+1}_0(x)|= &\; \int_{\R^d}\rho^{n+1}_1(x)-\rho^{n+1}_0(x)\\
   = &\;\int_{\R^d}\rho_1(x)-\rho_0(x)+\tau\sum_{k=0}^{n}\int_{\R^d}( \mu_1^{k+1}(x)-\mu_0^{k+1}(x))\, dx.
   \end{split}
  \]
Due to \eqref{eq:pd_relation_growth} we have
  \[
  \mu_1^{k+1}(x)-\mu_0^{k+1}(x)=\rho^k_1(x)G(q_1^{k+1}(x),x)-\rho^k_0(x)G(q_0^{k+1}(x),x)
  \]
  where $q_i^{k+1}=(p_i^{k+1})^c.$ We now claim that
  \begin{equation}\label{claim0010}
  q^{k+1}_0(y)\leq q_1^{k+1}(y) \quad \textrm{a.e.} \; y\in \spt\rho_1^k.
  \end{equation}
 Indeed, if this were not the case, then for some $y\in\spt\rho_1^k$ we have
  \[
  q_1^{k+1}(y)=p^{k+1}_1(T_{p_1^{k+1}}(y))+\frac{1}{2\tau}|T_{p_1^{k+1}}(y)-y|^2<p^{k+1}_0(T_{p_0^{k+1}}(y))+\frac{1}{2\tau}|T_{p_0^{k+1}}(y)-y|^2=q_0^{k+1}(y).
  \]
By the comparison principle, $p^{k+1}_0(x)\leq p^{k_1+1}_1(x)$ for almost all $x\in\spt\rho_1^{k+1}$.   Since $T_{p_1^{k+1}}(y)\in \spt\rho_1^{k+1}$ for every $y\in \spt\rho_1^k$, we can conclude that $p^{k+1}_0(T_{p_1^{k+1}}(y))\leq p^{k+1}_1(T_{p_1^{k+1}}(y))$ for almost every $y\in \spt\rho_1^k$.  Now we can compute
  \[
  \begin{split}
  p^{k+1}_0(T_{p_1^{k+1}}(y))+\frac{1}{2\tau}|T_{p_1^{k+1}}(y)-y|^2\leq &\; p^{k+1}_1(T_{p_1^{k+1}}(y))+\frac{1}{2\tau}|T_{p_1^{k+1}}(y)-y|^2\\
  <&\;p^{k+1}_0(T_{p_0^{k+1}}(y))+\frac{1}{2\tau}|T_{p_0^{k+1}}(y)-y|^2
  \end{split}
  \]
 which contradicts the optimality of $T_{p_0^{k+1}}$.   Thus \eqref{claim0010} holds, and since $G$ is non-increasing,
  \[
  \rho^k_1(x)G(q_1^{k+1}(x),x)-\rho^k_0(x)G(q_0^{k+1}(x),x)\leq \big(\rho^k_1(x)-\rho^k_0(x)\big)G(q_1^{k+1}(x),x).
  \]
 At last, we obtain
  \[
  \int_{\R^d}|\rho^{n+1}_1(x)-\rho^{n+1}_0(x)|\leq  \int_{\R^d}|\rho_1(x)-\rho_0(x)|+\tau B\sum_{k=0}^{n}  \int_{\R^d}|\rho^{k}_1(x)-\rho^{k}_0(x)|.
  \]
  Or in terms of $\rho^{\tau}_i$,
     \[
 \int_{\R^d}|\rho^\tau_1(x,t)-\rho^\tau_0(x,t)|\, dx\leq (1+\tau B)  \|\rho_1-\rho_0\|_{L^1(\R^d)}+B\int_0^t \int_{\R^d}|\rho^{\tau}_1(x,s)-\rho^{\tau}_0(x,s)| \,dx\,ds.
 \]
 Now one can conclude via Gronwall's inequality.
 \medskip

In general, let $\rho_{\dag}(x) := \min(\rho_0(x),\rho_1(x))$.
By the assumption of $s$, $\rho_\dag$ satisfies \eqref{good_data_0} and $\rho_\dag \in X$.
Hence, applying \eqref{contraction} to the pairs $\rho_i$ and $\rho_\dag$ yields that
\[
\int_0^t\int_{\R^d}|\rho^\tau_1(x,s)-\rho^\tau_0(x,s)|\, dx \,ds
\leq \frac{1}{B} (e^{Bt}-1)(1+\tau B) \big(\|\rho_1-\rho_\dag\|_{L^1(\R^d)}+\|\rho_0-\rho_\dag\|_{L^1(\R^d)}\big).
\]
Thanks to the definition of $\rho_\dag$, the right hand side exactly gives the desired bound.
\end{proof}

\begin{proof}[Proof of Proposition \ref{equicontinuity}]

For any $\delta>0$, we may mollify $\rho_0$ to obtain a $\rho_1\in BV(\R^d)$, such that $\rho_1\in X$ satisfies \eqref{good_data_0} and
$\norm{\rho_0 - \rho_1}_{L^1(\R^d)}\leq \delta$.

Let $\rho^\tau$ and $\rho^\tau_1$ be as given in \eqref{interpolation} with initial data $\rho_0$ and $\rho_1$, respectively.  By Lemma~\ref{lem: contraction}, 
\[
\begin{split}
&\;\int_0^T \int_{\R^d}|\rho^{\tau}(x+\epsilon y,t)-\rho^{\tau}(x,t)|\, dx\,dt\\
\leq &\; 2\int_0^T\int_{\R^d} |\rho^\tau(x,t)-\rho_1^{\tau}(x,t)|\, dx\,dt+\int_0^T \int_{\R^d}|\rho^{\tau}_1(x+\epsilon y,t)-\rho^{\tau}_1(x,t)|\, dx\,dt\\
\leq &\;\frac{2}{B}(e^{BT}-1)(1+\tau B)\norm{\rho_0- \rho_1}_{L^1(\R^d)} + T\epsilon |y|\sup_{0\leq t\leq T} \norm{\rho_1^\tau(\cdot,t)}_{BV(\R^d)}.
\end{split}
\]
In the last inequality, we applied the standard $L^1$-Lipschitz property of BV functions.
Thanks to Corollary \ref{BV2}, 
$$
\lim_{\epsilon\to 0}\, \sup_{0<\tau\leq (2B)^{-1}}\, \int_0^T \int_{\R^d} |\rho^{\tau}(x+\epsilon y,t)-\rho^{\tau}(x,t)|\,dx\,dt \leq C\delta,
$$
where $C$ depends on $B$ and $T$.
Since $\delta$ is arbitrary, we can now conclude.
\end{proof}

Lemma \ref{time_density} and Proposition \ref{equicontinuity} together yield the strong convergence of $\rho^{\tau}$ in $L^1([\R^d\times [0,T])$.

\begin{corollary}\label{L^1_conv}
Along a subsequence, $\rho^{\tau}$ strongly converges to some $\rho$ in $L^1(\R^d\times [0,T])$.
\end{corollary}

\begin{proof}
The convergence result follows from Proposition 5.6 in \cite{JKT}, which is based on the equicontinuity estimates Lemma~\ref{time_density} and Proposition~\ref{equicontinuity}.
\end{proof}

\section{Continuum Limit}
\label{sec: continuum limit}

Our goal is to show that these sequences converge to a very weak solution of the continuous-in-time problem $(P)$.
We will fix the time horizon $T>0$ and keep the assumptions on $\rho_0$ and $\Omega$ given in Section \ref{sec: equicontinuity}.
To emphasize the independence of solutions  on the choice of $\Omega$, we use the extended notion of the compactly supported solutions $(\rho^{\tau},p^{\tau})$ in $\R^d\times [0,T]$ with zero value for $x$ outside of $\Omega$.

\medskip

We begin by showing that the discrete solutions approximately solve the continuity equation.

\begin{lemma}\label{continuity}
Fix $T>0$.
Assume $\tau \leq \min\{1/B, T/2\}$.
The pair $(\rho^{\tau}, p^{\tau})$ approximately solves the  continuity equation in the weak sense, i.e., for all $\phi\in C^2_0(\R^d\times[0,T])$ and $t_0\in [2\tau, T]$,
\begin{equation}
\begin{split}
&\;\int_{0}^{t_0} \int_{\R^d} \rho^{\tau} \partial_t\phi \, dx\,dt+ \int_{0}^{t_0}\int_{\R^d}\mu^{\tau}\phi-\rho^{\tau}\nabla p^{\tau}\cdot \nabla \phi \, dx\,dt\\
=&\; \int_{\R^d} \rho^{\tau}(x,t_0)\phi(x,t_0) - \rho_0(x) \phi(x,0) \,dx +\epsilon_\tau.
\end{split}
\label{eqn: discrete weak formula almost equality}
\end{equation}
Here the error $\epsilon_\tau$ satisfies $|\epsilon_\tau|\leq \tau^{1/2}M$,
where $M$ is a constant depending on $T$, $G$, $s$, $\rho_0$ and $\phi$.
\end{lemma}

\begin{proof}
From the definition of our interpolations, we have
\begin{equation}
\begin{split}
&\;\int_{0}^{t_0-\tau} \int_{\R^d} \rho^{\tau}(x,t)\frac{\phi(x,t+\tau)-\phi(x,t)}{\tau}\,dx\,dt\\
=&\;-\int_{\tau}^{t_0-\tau} \int_{\R^d} \frac{\rho^{\tau}(x,t)-\rho^{\tau}(x,t-\tau)}{\tau}\phi(x,t)\,dx\,dt\\
&\;+\frac{1}{\tau}\int_{t_0-\tau}^{t_0}\int_{\R^d} \rho^{\tau}(x,t-\tau)\phi(x,t) \, dx\,dt -\frac{1}{\tau}\int_{0}^{\tau}\int_{\R^d} \rho^{\tau} \phi \,dx\,dt.
\end{split}
\label{eqn: summation by part}
\end{equation}

For the left hand side, by Taylor expansion,
\[
\begin{split}
&\;\left|\int_{0}^{t_0-\tau} \int_{\R^d} \rho^{\tau}(x,t)\frac{\phi(x,t+\tau)-\phi(x,t)}{\tau}\,dx\,dt - \int_{0}^{t_0} \int_{\R^d} \rho^{\tau} \partial_t\phi \,dx\,dt\right|\\
\leq &\;\int_{0}^{t_0-\tau} \rho^{\tau}(\R^d,t)\cdot\frac{\tau}{2}\|\partial_t^2 \phi\|_{L^\infty( \R^d\times [0,T])}\,dt
+\int_{t_0-\tau}^{t_0}\rho^\tau(\R^d,t)\|\partial_t \phi\|_{L^\infty(\R^d\times [0,T])}\,dt\\
\leq &\;\tau e^{B(T+\tau)}\rho_0(\R^d)\|\phi\|_{C^2(\R^d\times[0,T])}.
\end{split}
\]

For the first term on the right hand side of \eqref{eqn: summation by part},
by the pushforward formula,
\[
\begin{split}
&\;\int_{\tau}^{t_0-\tau}\int_{\R^d} \frac{\rho^{\tau}(x,t)-\rho^{\tau}(x,t-\tau)}{\tau}\phi(x,t)\,dx\,dt\\
=&\;
\int_{\tau}^{t_0-\tau}\int_{\R^d} \rho^{\tau}(x,t)\frac{\phi(x,t)-\phi\big(x+\tau\nabla p^{\tau}(x,t),t\big)}{\tau}+\mu^{\tau}(x,t)\phi(x,t)\,dx\,dt.
\end{split}
\]
Thanks to the Taylor expansion of $\phi\big(x+\tau\nabla p^{\tau}(x,t),t\big)$ and Corollary \ref{edi},
\[
\begin{split}
&\;\left|\int_{\tau}^{t_0-\tau}\int_{\R^d} \frac{\rho^{\tau}(x,t)-\rho^{\tau}(x,t-\tau)}{\tau}\phi(x,t)\,dx\,dt
-\int_{\tau}^{t_0-\tau}\int_{\R^d} -\rho^{\tau}\nabla p^\tau\cdot \nabla \phi+\mu^{\tau}\phi\,dx\,dt\right|\\
\leq &\; \|\phi\|_{C^2(\R^d\times[0,T])}\int_{0}^{T}\int_{\R^d} \frac{\tau}{2}\rho^{\tau}|\nabla p^\tau|^2\\
\leq &\; \tau M\|\phi\|_{C^2(\R^d\times[0,T])},
\end{split}
\]
where $M$ is a constant depending on $T$, $G$, $s$ and $\rho_0$.
From the Cauchy-Schwarz inequality and Corollary \ref{edi},
\[
\begin{split}
&\;\left|\left(\int_{0}^{\tau}+\int_{t_0-\tau}^{t_0}\right)\int_{\R^d} -\rho^{\tau}\nabla p^\tau\cdot \nabla \phi+\mu^{\tau}\phi\,dx\,dt\right|\\
\leq &\;\left[\left(\int_{0}^{\tau}+\int_{t_0-\tau}^{t_0}\right)\rho^\tau(\R^d,t)\,dt
\cdot\int_{0}^T\int_{\R^d}\rho^\tau|\nabla p^\tau|^2\,dx\,dt\right]^{1/2}\|\phi\|_{C^1( \R^d\times [0,T])}\\
&\;+ B\|\phi\|_{C(\R^d\times [0,T])}\left(\int_{-\tau}^{0}+\int_{t_0-2\tau}^{t_0-\tau}\right)\rho^\tau(\R^d,t)\,dt\\
\leq &\;\tau^{1/2} M\|\phi\|_{C^1(\R^d\times [0,T]}.
\end{split}
\]

For the last two terms in \eqref{eqn: summation by part}, we derive that
\[
\begin{split}
&\;\left|\frac{1}{\tau}\int_{t_0-\tau}^{t_0}\int_{\R^d} \rho^{\tau}(x,t-\tau)\phi(x,t) \, dx\,dt -\int_{\R^d} \rho^{\tau}(x,t_0)\phi(x,t_0) \, dx\right|\\
\leq &\;\frac{1}{\tau}\int_{t_0-\tau}^{t_0}\rho^{\tau}(\R^d,t-\tau) \|\phi(\cdot,t)-\phi(\cdot,t_0)\|_{C(\R^d)}\,dt\\
&\;+\frac{1}{\tau}\int_{t_0-\tau}^{t_0}\|\rho^{\tau}(\cdot,t-\tau)-\rho^{\tau}(\cdot,t_0)\|_{H^{-1}(\R^d)} \|\phi(\cdot,t_0)\|_{H^1( \R^d)}\,dt.
\end{split}
\]
The first term above is bounded by
$\tau e^{BT}\rho_0(\R^d)\|\phi\|_{C^1(\R^d\times [0,T])}$.
By the Cauchy-Schwarz inequality and the definition of $\rho^\tau$, the second term is bounded by
\[
\begin{split}
&\;C\|\phi\|_{C^1(\R^d\times[0,T])}\tau^{-1/2}\left(\int_{t_0-\tau}^{t_0}\|\rho^{\tau}(\cdot,t-\tau)-\rho^{\tau}(\cdot,t_0)\|_{H^{-1}(\R^d)}^2\,dt\right)^{1/2}\\
\leq &\;C\|\phi\|_{C^1(\R^d\times [0,T])}\Big(\|\rho^{\tau}(\cdot,t_0-2\tau)-\rho^{\tau}(\cdot,t_0-\tau)\|_{H^{-1}(\R^d)}^2
+\|\rho^{\tau}(t_0-\tau,\cdot)-\rho^{\tau}(t_0,\cdot)\|_{H^{-1}(\R^d)}^2\Big)^{1/2}.
\end{split}
\]
Here $C$ depends on the size of $\spt\phi$.
By Lemma \ref{time_density} (with $T$ there taken to be greater than or equal to $t_0+\tau$, say $2T$), this is further bounded by $CM\|\phi\|_{C^1(\R^d\times[0,T])}\tau^{1/2}$,
where $M$ depends on $T$, $G$, $s$ and $\rho_0$.
Hence,
\[
\left|\frac{1}{\tau}\int_{t_0-\tau}^{t_0}\int_{\R^d} \rho^{\tau}(x,t-\tau)\phi(x,t) \, dx\,dt -\int_{\R^d} \rho^{\tau}(x,t_0)\phi(x,t_0) \, dx\right|\\
\leq CM\|\phi\|_{C^1(\R^d\times [0,T])}\tau^{1/2}.
\]
Similarly, the last term in \eqref{eqn: summation by part} satisfies
\[
\left|\frac{1}{\tau}\int_{0}^{\tau}\int_{\R^d} \rho^{\tau}(x,t)\phi(x,t) \, dx\,dt -\int_{\R^d} \rho_0(x)\phi(x,0) \, dx\right|\\
\leq CM\|\phi\|_{C^1(\R^d\times [0,T])}\tau^{1/2}.
\]
Summarizing all the above estimates, we complete the proof.
\end{proof}

To send $\tau\to 0$ in \eqref{eqn: discrete weak formula almost equality} to obtain the continuum weak equation, we need to discuss the convergence of $\mu^\tau$ as $\tau\to 0$.

\begin{lemma}\label{source_conv}
Fix $T>0$.
For any $\phi\in L^\infty([0,T]; W^{1,\infty}(\R^d))$,
\[
\lim_{\tau\to 0} \int_0^T\int_{\R^d}\big(\mu^{\tau}(x,t)-\rho^{\tau}(x,t)G(p^{\tau}(x,t),x)\big)\phi(x,t)\,dx\,dt=0.
\]
\end{lemma}
\begin{proof}
Assume $\tau \leq 1/B$ and let $0\leq n\leq \lfloor \frac{T}{\tau}\rfloor$.
Take an arbitrary test function $\varphi\in W^{1,\infty}(\R^d)$.
Due to \eqref{eq:pd_relation_pushforward} and \eqref{eq:pd_relation_growth}, we have
\[
\int_{\R^d}\varphi(x)(\rho^n+\tau\mu^{n+1})G(p^c_{n+1}(x),x)\,dx
=\int_{\R^d}\rho^{n+1}(x)\varphi(T_{p_{n+1}}^{-1}(x))G\big(p^c_{n+1}(T_{p_{n+1}}^{-1}(x)),T_{p_{n+1}}^{-1}(x)\big)\,dx.
\]
Hence,
\begin{equation}
\begin{split}
&\;\int_{\R^d}\varphi(x)\big(\rho^n(x)G(p^c_{n+1}(x),x)-\rho^{n+1}(x)G(p_{n+1}(x),x)\big)\,dx\\
=&\;
\int_{\R^d}\rho^{n+1}(x)\Big(\varphi(T_{p_{n+1}}^{-1}(x))G\big(p^c_{n+1}(T_{p_{n+1}}^{-1}(x)),T_{p_{n+1}}^{-1}(x)\big)-\varphi(x)G\big(p_{n+1}(x),x\big)\Big)\,dx\\
&\;-\tau\int_{\R^d} \varphi(x)\mu^{n+1}(x)G(p^c_{n+1}(x),x)\, dx
\end{split}
\label{eqn: error in approximating the growth}
\end{equation}
The last term is trivially bounded by $\tau B^2 e^{B(T+\tau)}\rho_0(\R^d)\norm{\varphi}_{L^\infty(\R^d)}$.

To handle the first term, we have by the definition of $T_{p_{n+1}}$ and the $c$-transform that
\[
p^c_{n+1}(T_{p_{n+1}}^{-1}(x))=p_{n+1}(x)+\frac{1}{2\tau}|x-T_{p_{n+1}}^{-1}(x)|^2=p_{n+1}(x)+\frac{\tau}{2}|\nabla p_{n+1}(x)|^2.
\]
Recall that $g_0:=\norm{G}_{W^{1,\infty}\big([0,P_1]\times\R^d\big)}$ is defined in Lemma \ref{BV}.
Hence,
\[
\begin{split}
&\;\Big| \int_{\R^d}\rho^{n+1}(x)\Big(\varphi(T_{p_{n+1}}^{-1}(x))G\big(p^c_{n+1}(T_{p_{n+1}}^{-1}(x)),T_{p_{n+1}}^{-1}(x)\big) -\varphi(x)G\big(p_{n+1}(x),x\big)\Big)\Big|\\
\leq&\;
\tau \Big(B \norm{\nabla \varphi}_{L^\infty(\R^d)}+g_0\norm{\varphi}_{L^\infty(\R^d)}\Big) (\rho^{n+1}(\R^d))^{1/2}\norm{\nabla p_{n+1}}_{L^2(\rho^{n+1})}\\
&\;+\frac{\tau}{2}g_0\norm{\varphi}_{C(\R^d)}\norm{\nabla p_{n+1}}_{L^2(\rho^{n+1})}^2.
\end{split}
\]
Combining this with \eqref{eqn: error in approximating the growth}, we obtain by the Cauchy-Schwarz inequality that
\[
\left|\int_{\R^d}\varphi(x)\big(\mu^{n+1}(x)-\rho^{n+1}(x)G(p_{n+1}(x),x)\big)\,dx\right|\\
\leq \tau M \norm{\varphi}_{W^{1,\infty}(\R^d)}\big(1+\|\nabla p_{n+1}\|_{L^2(\rho^{n+1})}^2\big).
\]
where $M$ is a constant depending on $T$, $G$, $s$ and $\rho_0$.

Rewriting the above inequality in terms of $\rho^\tau$, $\mu^\tau$ and $p^\tau$, and replacing $\varphi(x)$ into $\phi(x,t)$, we take time integral over $[0,T]$ to find that
\[
\begin{split}
&\;\left|\int_0^T\int_{\R^d}(\mu^{\tau}(x,t)-\rho^{\tau}(x,t)G(p^\tau(x,t),x))\phi(x,t)\,dx\,dt\right|\\
\leq&\;
\tau M \norm{\phi(t,\cdot)}_{L^\infty([0,T];W^{1,\infty}(\R^d))}\left(T+\int_0^{T'}\int_{\R^d}\rho^\tau|\nabla p^{\tau}|^2\,dx\,dt\right),
\end{split}
\]
where $T' = \big(\lfloor\frac{T}{\tau}\rfloor+1\big)\tau$ is defined as in Lemma \ref{lem: energy dissipation inequality}.
Now we may conclude the proof by Corollary \ref{edi}.

\end{proof}

Now we are ready to prove the convergence result as $\tau\to 0$. It should be noted that the convergence of the pressure variable is established only with the positive part of $p^{\tau}$. This is due to the lack of information on the pressure away from the support of $\rho^{\tau}$. On the other hand $p^{\tau}$ is nonnegative on the support of $\rho^{\tau}$, and thus we do not lose information by this reduction.

\begin{prop}\label{gen_conv}
Let $\rho_0\in X$ satisfy \eqref{good_data_0}. Then $\rho^{\tau}$ is uniformly bounded in $L^{\infty}(\R^d\times [0,T])$ and strongly converges in $L^1(\R^d\times [0,T])$ along a subsequence to some $\rho$ in $L^{\infty}(\R^d\times [0,T])$.
Furthermore, $p^{\tau}_+ = \max(p^\tau,0)$ is uniformly bounded in $L^{\infty}(\R^d\times [0,T])$, and weak-$*$ converges along a subsequence to some $p$ in $L^\infty(\R^d\times [0,T])$.
Moreover, along a subsequence, $\rho^{\tau}p^{\tau}$ and $s^*(p^{\tau})$ converge weakly in $L^1(\R^d\times[0,T])$ to $\rho p$ and $s^*(p)$, respectively.
\end{prop}

\begin{proof}

Since $p^{\tau}$ has a uniform upper bound in $\R^d\times[0,T]$, by \eqref{eq:pd_relation_energy}, $\rho^{\tau}$ is uniformly bounded in $L^{\infty}(\R^d\times [0,T])$, and so is $\rho^\tau p^\tau$.
Hence, along a subsequence, $p^{\tau}_+$ converges to some $p\in L^{\infty}\big(\R^d\times[0,T]\big)$ in the weak-$*$ topology.

By Corollary~\ref{L^1_conv}, it follows that $\rho^{\tau} p^{\tau} = \rho^{\tau} p^{\tau}_+$ converges weakly to $\rho p$ in $L^1( \R^d\times[0,T])$ up to a further subsequence.

\medskip

Next we show that $p \in\partial s(\rho)$.  By Lemma~\ref{dual_relation}, it suffices to show that
\begin{equation}\label{continuum_dual}
 \rho(x,t)p(x,t)= s(\rho(x,t))+s^*(p(x,t))\quad \textup{for a.e.} \; (x,t)\in \R^d\times[0,T].
 \end{equation}
 It is enough to show that the left hand side is greater or equal to the right, since the other inequality is always true by definition.
   From the discrete scheme and the fact $s^*(p) = 0$ for all $p\leq 0$, we have
 
 \begin{equation}\label{discrete_dual}
 \rho^{\tau}(x,t)p^{\tau}(x,t)=s(\rho^{\tau}(x,t))+s^*(p^{\tau}_+(x,t))\quad a.e.\;(x,t)\in \R^d\times[0,T].
 \end{equation}

Since along a subsequence $\rho^{\tau}\to \rho$ in $L^1(\R^d\times [0,T])$, the same holds for $s(\rho^{\tau})$ due to the continuity of $s$. The desired inequality follows by combining this fact with the weak convergence of $\rho^{\tau}p^{\tau}$ to $\rho p$ in $L^1(\R^d\times [0,T])$, the weak-$*$ convergence of $p^{\tau}_+$ in $L^{\infty}(\R^d\times [0,T])$ (both along a subsequence),  and the weak lower semi-continuity of $s^*$.

It remains to show that $s^*(p^{\tau})$, or equivalently $s^*(p^{\tau}_+)$, weakly converges to $s^*(p)$ along a subsequence in $L^1(\R^d\times [0,T])$.
This immediately follows from \eqref{continuum_dual} and \eqref{discrete_dual}, noting that $\rho^\tau p^{\tau}\rightharpoonup\rho p$ and $s(\rho^{\tau})\to s(\rho)$ in $\R^d\times [0,T]$ along a subsequence.
 \end{proof}

Now we are ready to characterize the continuum limit as a very weak solution of the diffusion equation \eqref{weak00}.
Recall from Lemma~\ref{lem:linf_bound} that $p^\tau \leq  M_0:= \max(b_1,\inf \partial s(\norm{\rho_0}_{L^\infty(\R^d)}))$ for all $\tau>0$.

\begin{theorem}\label{continuum}
Suppose that either $s\in C^1_{loc}([0,+\infty))$, or $G(\cdot, x)$ is affine on $[0, M_0]$ for all $x\in \R^d$. Then for any $T>0$, the limit density and pressure $(\rho, p)$ given in Proposition \ref{gen_conv} satisfy
\begin{equation}\label{formula}
\int_0^{t_0} \int_{\R^d} \rho\partial_t \phi +  s^*(p) \Delta \phi +G(p,x)\rho\phi\, dx\, d\tau = \int_{\R^d} \rho(x,t_0)\phi(x,t_0) \,dx- \int_{\R^d} \rho_0(x) \phi(x,0) \,dx.
\end{equation}
for any $\phi\in C^{\infty}([0,\infty)\times \R^d)$ and a.e.\;$t_0\in [0,T]$.
\end{theorem}
\begin{proof}
We first show that, along a subsequence as $\tau\to 0$,
\begin{equation}\label{last}
\int_{0}^{t_0} \int_{\R^d} G(p^{\tau}, x)\rho^{\tau} \phi \,dx \,dt \to \int_{0}^{t_0} \int_{\R^d} G(p,x)\rho\phi \,dx \,dt
\end{equation}
for any $\phi\in C^{\infty}_0([0,\infty)\times \R^d)$ and any $t_0\in [0,T]$.
If $s\in C^1_{loc}([0,\infty))$, we have $p^{\tau}_+ = s'(\rho^{\tau})$ almost everywhere.
By Corollary~\ref{L^1_conv}, along a subsequence, $p^\tau_+$ a.e.\;converges to $p$ on $\R^d\times [0,T]$. Since $G$ is continuous, $\rho^\tau G(p^{\tau}(x,t),x)$ a.e.\;converges to $\rho(x,t)G(p(x,t),x)$.  Moreover, since $G$ is uniformly bounded, we can conclude \eqref{last} by the dominated convergence theorem.
Otherwise, if $G(\cdot,x)$ is affine, \eqref{last} holds because of $L^1$-convergence of $\rho^{\tau}$ to $\rho$ and the weak-$*$ convergence of $p^{\tau}_+$ to $p$ in $L^{\infty}(\R^d\times [0,T])$.

We then claim that
\begin{equation}\label{eqn: very weak form}
\int_{\R^d} \rho^{\tau}\nabla p^{\tau}\cdot \nabla \phi \,dx = -\int_{\R^d} s^*(p^\tau)\Delta\phi \,dx.
\end{equation}
Given \eqref{eq:pd_relation_energy}, if $s^*\in C^1_{loc}(\R)$, this is trivial by integration by parts. Suppose not. We construct $\{s_\e^*\}_{\e>0}$ to be a non-negative sequence of $C^1$-approximation of $s^*$ as follows.
Take $\zeta\in C_0^\infty(\R)$, such that $\zeta\geq 0$, $\int_\R\zeta = 1$ and $\spt \zeta\in[0,1]$.
Define $\zeta_\e(x) : = \e^{-1}\zeta(x/\e)$ and let $s_\e^* = s^* * \zeta_\e$.
Since $s^*$ is non-decreasing on $\R$ and locally Lipschitz, $\{s_\e^*\}_\e$ satisfies that
$s_\e^*$ is decreasing in $\e$ and $s_\e^*\to s^*$ locally uniformly as $\e\to 0$.
Moreover, since $s^*$ in convex, $(s_\e^*)'(\cdot)$ is decreasing in $\e$, and at any differentiable point of $s^*$, $(s_\e^*)'\nearrow(s^*)'$ as $\e \to 0$.
Let $\rho^{\tau}_\e := (s_\e^*)'(p^\tau)$.
Then with $\Omega$ being the sufficiently large convex smooth domain used to construct $(\rho^\tau,p^\tau)$, which contains $\spt \rho^\tau$ and $\spt\rho^\tau_\e$,
\begin{equation}\label{above}
\int_{\Omega} \rho_\e^\tau\nabla p^\tau \nabla \phi \,dx= -\int_{\R^d} s_\e^*(p^\tau)\Delta \phi \,dx\to -\int_{\R^d}  s^*(p^\tau)\Delta \phi \,dx 
\end{equation}
as $\e\to 0$. 

The convergence is because of the local uniform convergence from $s_\e^*$ to $s^*$.
On the other hand, since $s^*$ is convex, it has countably many non-differentiable points, which are denoted as $\{a_i\}_{i = 1}^\infty$.
Let $A_i = \{x\in \Omega:\; p^\tau = a_i\}$ and $A = \cup_{i = 1}^\infty A_i$.
It is known that $\nabla p^\tau = 0$ a.e.\;in $A_i$ for all $i$.
Hence,
\[
\begin{split}
\left|\int_{\Omega} (\rho_\e^\tau-\rho^\tau)\nabla p^\tau \nabla \phi \,dx\right|
\leq &\;\int_{\Omega\backslash A} |\rho_\e^\tau-\rho^\tau||\nabla p^\tau| |\nabla \phi| \,dx\\
\leq &\;\|\nabla p^\tau\|_{L^\infty(\Omega)} \|\nabla \phi\|_{L^\infty(\R^d)}\int_{\Omega\backslash A} |(s_\e^*)'(p^\tau)-(s^*)'(p^\tau)| \,dx.
\end{split}
\]
Since $p^\tau$ is $c$-concave, $\|\nabla p^\tau\|_{L^\infty(\Omega)}$ admits a uniform bound.
Hence, by the monotone convergence $(s_\e^*)'(p^\tau)\nearrow (s^*)'(p^\tau)$ on $\Omega\backslash A$, the right hand side above goes to $0$ as $\e \to 0$.
Combining this with \eqref{above}, we finish the proof of \eqref{eqn: very weak form}.

Now by Proposition \ref{gen_conv}, for any $0<t_0\leq T$ and any $\phi\in C^{\infty}_0( \R^d\times [0,T])$,
\begin{equation}\label{first}
\int_0^{t_0} \int_{\R^d} \rho^{\tau}\nabla p^{\tau}\cdot \nabla \phi \,dx \,dt \to -\int_0^{t_0} \int_{\R^d} s^*(p)\Delta\phi \,dx \,dt
\end{equation}
along a subsequence as $\tau\to 0$.

Furthermore, by Corollary~\ref{L^1_conv} and Fubini's theorem, $\rho^\tau(t_0,\cdot) \to \rho(t_0,\cdot)$ in $L^1(\R^d)$ for a.e.\;$t_0\in [0,T]$.

Therefore, \eqref{formula} follows from Lemma~\ref{continuity}, Lemma~\ref{source_conv}, Proposition~\ref{gen_conv}, \eqref{last} and \eqref{first}.
\end{proof}

\begin{remark}
Combining Lemma \ref{source_conv} with \eqref{last}, we obtain that, along a subsequence, $\mu^\tau \rightharpoonup \rho G(p,x)$ in $L^1([0,T];W^{-1,1}(\R^d))$ as $\tau \to 0$.
\end{remark}

 When $s=s_{\infty}$, similar to \cite{PQV}, we obtain strong convergence of $p^{\tau}$ using the monotonicity property established in Corollary \ref{pressure_monotone}.

\begin{theorem}\label{continuum2}
Let $s=s_{\infty}$ and let $\rho_0(x) \in[0, 1]$. Then $\{\rho^{\tau}\}_{\tau>0}$ is uniformly bounded in $L^{\infty}(\R^d\times [0,T])$ and converges to some $\rho$ in $L^1(\R^d\times [0,T])$ along a subsequence. Moreover, there exists $p\in \partial s(\rho)$ such that $p^{\tau}_+\to p$ in $L^1(\R^d\times [0,T])$ and $p^{\tau}_+ \rightharpoonup p$ in $L^2([0,T];H^1( \R^d))$ along a subsequence.
Lastly, $(\rho,p)$ satisfies $(\rho-1)p = 0$ a.e.\;and
\begin{equation}\label{formula2}
\int_0^{t_0} \int_{\R^d} \rho\partial_t \phi - \nabla p\cdot\nabla \phi +G(p,x)\rho\phi\, dx\, dt =  \int_{\R^d} \rho(t_0,x)\phi(t_0,x) \,dx- \int_{\R^d} \rho_0(x) \phi(0,x) \,dx
\end{equation}
for any $\phi\in C^{\infty}([0,\infty)\times \R^d)$ and a.e.\;$t_0\in [0,T]$.
\end{theorem}

\begin{proof}
As before, $p^{\tau}_+$ is uniformly bounded in $L^{\infty}(\R^d\times [0,T])$ (see Lemma \ref{lem:linf_bound}). Also, since $(\rho^{\tau}-1)p^{\tau}_+=0$ a.e.\;from their dual relation, for $t\in [0,T]$, $p^\tau_+$ is supported on a compact set that is uniform in $\tau$ (see Section \ref{sec: finite speed propagation}).
Moreover for any $t_0\in [0,T]$,
\begin{equation}\label{pressure_H1}
\int_0^{t_0} \int_{\R^d} |\nabla p^{\tau}_+|^2 dx = \int_0^{t_0}\int_{\R^d} \rho^{\tau}|\nabla p^{\tau}_+|^2 dx,
\end{equation}
which is uniformly bounded with respect to $\tau$ due to Corollary~\ref{edi}.

\medskip

The convergence of $p^{\tau}_+$ in $L^1(\R^d\times [0,T])$ is a consequence the time-monotonicity in Corollary~\ref{pressure_monotone}. In fact, if we consider the linear interpolation
$$
\tilde{p}^{\tau}(x,(n -1+ \theta)\tau):= \theta p^{\tau}_+(x,n\tau) + (1-\theta) p^{\tau}_+(x,(n-1)\tau) \hbox{ for } 0\leq \theta<1\hbox{ and }n\in \mathbb{N}_+,
$$
then for any $N\leq T/\tau+1$,
$$
\int_0^{N\tau} \int_{\R^d}  |\partial_t\tilde{p}^{\tau}|\, dx\, dt = \sum_ {n=1}^N \tau \int_{\R^d}\left|\frac{p^{\tau}_+(x,n\tau) - p^{\tau}_+(x,(n-1)\tau)}{\tau}\right|\,  dx= \int_{\R^d} (p^{\tau}_+(x,N\tau) - p^{\tau}_+(x,0)) \,dx,
$$
and thus $\partial_t\tilde{p}^{\tau} \in L^1(\R^d\times [0,T])$ for any $T>0$.  Combined with \eqref{pressure_H1} we conclude that $\tilde{p}^{\tau}$, and thus $p^{\tau}_+$, strongly converges to $p$ in $L^1(\R^d\times [0,T])$ along a subsequence. Since $p^{\tau}_+$ and thus $p$ are uniformly bounded in $L^{\infty}(\R^d\times [0,T])$, this convergence also holds in $L^2(\R^d\times [0,T])$. Combined with the uniform $L^2$-bound of $\nabla p^{\tau}_+$ by \eqref{pressure_H1}, it follows that $\nabla p^{\tau}_+ \rightharpoonup\nabla p$ in $L^2(\R^d\times [0,T])$.

\medskip

Lastly, the dual relation $p(\rho-1)=0$ is  obtained from the discrete version $p^{\tau}_+\rho^{\tau} = p^{\tau}_+$, from the strong convergence of $p^{\tau}_+$ and $\rho^{\tau}$ in $L^1\cap L^2(\R^d\times [0,T])$. Finally, \eqref{formula2} can be justified as in the proof of Theorem~\ref{continuum}. 
\end{proof}

\section{Coincidence of Solutions}
\label{sec: coincidence of solutions}

In this section we show that our continuum limit solutions in many cases, including those for $s=s_m$ and $s=s_{\infty}$ with general $G$, are sufficiently regular to coincide with the existing notion of unique solutions.

\subsection{Regular energy}

\begin{definition}
$(\rho,p)$ is a {\it very weak} solution of \eqref{weak00} if they are nonnegative, compactly supported and bounded functions in $\R^d\times [0,T]$ such that
  $s^*(p)\in L^2(\R^d\times [0,T])$ and satisfies \eqref{continuum_dual} and \eqref{formula}.
\end{definition}

\begin{theorem}
Let $s \in C^1_{loc}([0,\infty))$ and suppose that for any $C>0$ there exists a constant $M=M_C$ such that
\begin{equation}\label{extra}
x|s'(x)-s'(y)| \leq M |x-y| \quad  \hbox{ for any } x,y \in[0, C].
\end{equation}
Then the continuum pair $(\rho,p)$ as given in Theorem~\ref{continuum} is the unique very weak solution of \eqref{weak00}.
\end{theorem}

\begin{remark}
Note that the assumptions are satisfied for $s(\rho) = \rho^m$ with $m>1$.
\end{remark}

\begin{proof}
From parallel arguments to Theorem 6.5 and Theorem 6.6 in \cite{vazquez} that uses the Hilbert duality method, where the same outline of proof applies when the Dirichlet data is replaced by the Neumann data,
we obtain the inequality
$$
\int_{\R^d} (\rho_1(t_0,x) - \rho_2(t_0,x))_+ \,dx \leq \int_0^{t_0}\int_{\R^d}  (\rho_1G(p_1,x) - \rho_2G(p_2,x))_+ \,dx \,dt.
$$
where $p_i = s'(\rho_i) $. Since $G(z,x)$ is bounded and is Lipschitz in $z$, we have
\begin{equation*}
\int_{\R^d} |\rho_1(t_0,x) - \rho_2(t_0,x)| \,dx \leq A\int_0^{t_0} \int_{\R^d} \big(\rho_1|p_1-p_2| + |\rho_1-\rho_2|\big)\,dx \,dt,
\end{equation*}
Where $A:= \sup_{z\leq \norm{p_1 +p_2}_{\infty}}  G'(z)$. Hence if we know that
\begin{equation*}
\int_{\R^d} \rho_1(x,t)|p_1(x,t)-p_2(x,t)| \,dx \leq M\int_{\R^d} |\rho_1(x,t) - \rho_2(x,t)| \,dx,
\end{equation*}
where $M$ is a uniform constant for $0\leq t\leq T$, then we can conclude by Gronwall's inequality. This is true due to Lemma~\ref{lem:linf_bound} and \eqref{extra}.
\end{proof}

\subsection{ Tumor growth model}

When $s=s_{\infty}$, the above argument does not apply since the pressure difference can no longer be bounded by the density difference. Instead, we resort to stronger notion of weak solutions with information on their time derivatives.

\begin{definition}
$(\rho,p)$ is a  weak solution of $(P)$ with $s=s_{\infty}$ if they are compactly supported functions in $\R^d\times [0,T]$ such that $0\leq \rho\leq 1$,  $p \in L^{\infty}(\R^d\times [0,T]) \cap L^2([0,T];H^1(\R^d))$, \eqref{continuum_dual} and \eqref{formula2}  hold, and in addition
 \begin{equation}\label{measure}
 \rho_t, p_t \in L^1(\R^d\times [0,T]).
 \end{equation}

\end{definition}

The continuum limit pair $(\rho,p)$ obtained in Theorem~\ref{continuum2} is a weak solution of $(P)$, with \eqref{measure} satisfied due to the monotonicity of $\rho$ and $p$ in time by virtue of Corollary~\ref{pressure_monotone}.

\medskip

The following theorem is obtained in \cite{PQV}, however we sketch their proof to highlight the necessary properties of weak solutions we need in the proof.

\begin{theorem}
Suppose $G(\cdot,x)$ is locally uniformly $C^2$. Then the continuum limit pair $(\rho,p)$ given in Theorem~\ref{continuum2} is the unique weak solution of \eqref{weak00} with $s=s_{\infty}$.
\end{theorem}

\begin{proof}
Let us consider two pairs of weak solutions $(\rho_i,p_i)_{i=1,2}$ of $(P)$ with $s=s_{\infty}$. Denote $\Omega_T:=\Omega\times [0,T]$, where $\Omega$ is sufficiently large so that $\Omega_T$ contains the support of $(\rho_i,p_i)$ for $i=1,2.$ Following \cite[Section 3]{PQV}, we write \eqref{formula2} in terms of the dual equation for $\psi$, i.e.
\begin{equation}\label{weak_2}
\int\int_{\Omega_T} (\rho_1-\rho_2 + p_1-p_2) [A\partial_t\psi + B\Delta \psi + AG(p_1,x) \psi - CB\psi]=0,
\end{equation}
where, since $\rho_i=1$ whenever $p_i>0$ and otherwise $\rho_i \leq 1$,
$$
A = \dfrac{\rho_1 - \rho_2}{(\rho_1-\rho_2) + (p_1-p_2)},   \quad B = \dfrac{p_1-p_2}{(\rho_1- \rho_2) + (p_1-p_2)}  \in [0,1],
$$
and
$$
0\leq C = - \rho_2\dfrac{G(p_1, x) - G(p_2,x)}{p_1-p_2} \leq M <\infty.
$$
Here $A$ is defined zero when $\rho_1=\rho_2$, and $B$ is defined zero when $p_1=p_2$.  In \cite{PQV} one applies Hilbert's duality method for the dual equation in \eqref{weak_2}. More precisely the idea is to solve the dual problem
$$
\left\{\begin{array}{ll}
A\partial\psi + B\Delta \psi + AG(p)\psi - CB\psi = A\Phi\hbox{ in } \Omega_T,\\
\psi=0 \hbox{ in } \partial\Omega\times (0,T), \quad \psi(\cdot,T)=0 \hbox{ in } \Omega.
\end{array} \right.
$$
for smooth $\Phi$, and use $\psi$ as test function in \eqref{weak_2}. This would yield uniqueness since $\Phi$ can be chosen arbitrarily. However the coefficients of the dual problem are neither smooth nor strictly positive, hence we need approximation arguments to derive uniqueness.

\medskip

For the approximation to create small error terms, we need some regularity assumptions on the coefficients. First the coefficients to be in $L^2(\Omega_T)$, which is fine since they are bounded. In addition we need $\nabla [G(p_i,x)] \in L^2(\Omega_T)$, $G(p_i,x)\in L^{\infty}(\Omega_T)$, and $\partial_t C \in L^1(\Omega_T)$.  Since $p_i$'s are bounded, it remains for us to check that
\begin{equation}\label{condition_C}
\nabla [G(p_i,x) ]\in L^2(\Omega_T) \hbox{ and } \partial_t C \in L^1(\Omega_T).
\end{equation}

The first bound follows from the fact that $G$ is locally Lipschitz, as well as the fact that $\nabla  p_i \in L^2 (\Omega_T)$. To check the second condition we write
$$
\dfrac{G(p_1,x) - G(p_2,x)}{p_1-p_2} = \int_0^1 G_p((1-s)p_1 + sp_2,x) \,ds.
$$
Thus
$$
C_t = -\rho_t \int_0^1 G_p((1-s)p_1 + sp_2,x) \,ds -\rho\int_0^1 G_{pp}((1-s)p_1 + sp_2,x) ((1-s) (p_1)_t + s(p_2)_t)\,ds.
$$
The first term is integrable since $\rho_t \in L^1(\Omega_T)$ and $\int_0^1 G_{pp}((1-s)p_1 + sp_2,x) \,ds$ is bounded due to the bound on $p_i$'s.

\medskip

Since $(p_1)_t, (p_2)_t \geq 0$, we conclude by Fubini's theorem that
$$
\int_{\Omega_T}  |C_t| \leq  M \int_{\Omega_T} [(p_1)_t + (p_2)_t]  \leq 2M \sum_{i=1,2}\norm{p_i}_{L^1(\Omega)}
$$
where
$$
M=\sup_{|p| \leq \max\{\norm{p_1}_{\infty}, \norm{p_2}_{\infty}\}, x\in \Omega} (|\partial_pG(p,x)| + |\partial_{pp} G(p,x)|).
$$

\end{proof}

\appendix
\section{An Improved Family of Barrier Densities}
\label{appendix: improved barriers}
We shall construct a refined version of barrier density $\rho_R$ discussed in Section \ref{sec: finite speed propagation}, to prove Proposition \ref{prop_iteration}.

Fix parameter $A\geq 1$; it will be clear later that $A$ roughly characterizes how steep the boundary transition of the density can be.
Let $Q$ be defined in \eqref{eqn: def of Q}, and let $w_1$ and $c_*$ be defined in Lemma \ref{lem: property of auxiliary function Q}.
We show that

\begin{lemma}\label{lem: property of auxiliary function Q_A}
\begin{enumerate}
  \item For all $A\geq 1$, there exists a unique $w_A\in(0,w_1]$ only depending on $\tilde{G}$, such that
\begin{equation*}
Q(w_A) =\frac{A^2\tau}{2}|Q'(w_A)|^2.
\end{equation*}
\item $w_A$ is decreasing in $A$ and $A|Q'(w_A)|$ is increasing in $A$.

  \item $$
  |Q'(w_A)|\leq \min\left\{ c_*, \sqrt{\frac{2z_M}{A^2\tau}}\right\}.
  $$

\end{enumerate}
\begin{proof}
The first two claim easily follow from the monotonicity of $Q$ and $|Q'|$ by Lemma \ref{lem: property of auxiliary function Q}.

The last one 
follows from Lemma \ref{lem: property of auxiliary function Q} and the fact that $Q(w_A)\leq z_M$.
.
\end{proof}
\end{lemma}

Take $R\geq w_A/A+1$, and  define
\begin{equation}
q_{R,A}(r) =
\begin{cases}
z_M& \mbox{ if }r\leq R-w_A/A,\\
Q(A(r-R+w_A/A))& \mbox{ if }r\in(R-w_A/A,R],\\
-\infty&\mbox{ otherwise.}
\end{cases}
\label{eqn: def of q_RA}
\end{equation}
Clearly, this generalizes \eqref{eqn: def of q_R}.
Let $\rho_{R,A}$ be defined by (c.f.\;\eqref{eqn: precise def of rho_R})
\begin{equation}
\rho_{R,A}(r)
\begin{cases}
= (s^*)'(z_M)& \mbox{ if }r\leq R-w_A/A,\\
\in  \frac{(1-\tau q_{R,A}''(r))(1-\tau r^{-1} q_{R,A}'(r))^{d-1}}{1+\tau \tilde{G}(q_{R,A}(r))}
\cdot \partial s^*\left(q_{R,A}(r) - \frac{\tau}{2}|q_{R,A}'(r)|^2\right) & \mbox{ if }r\in(R-w_A/A,R],\\
=0&\mbox{ otherwise.}
\end{cases}
\label{eqn: precise def of rho_RA}
\end{equation}
We know that $\rho_{R,A}$ is in a plateau-like shape, with the boundary transition from $0$ to the height of the plateau taking place within an annular region of width $w_A/A$.
Let $\rho_\dag$ be the optimal new density  corresponding to $\rho_{R,A}$ (see \eqref{eqn: the optimal new density}) in the modified problem \eqref{eqn: dual problem repeated}.
Then $\rho_{R,A}$ is supported on $\overline{B_R}$, while $\rho_\dag$ is supported on $\overline{B_{\tilde{R}}}$, where we define $\tilde{R}:=R+\tau A|Q'(w_A)|$ with abuse of notations.

In the spirit of Lemma \ref{lem: new density is dominated by another barrier}, one can readily show $\rho_\dag \leq \rho_{\tilde{R},A}$ almost everywhere.
However, we shall improve this by showing that $\rho_\dag \leq \rho_{\tilde{R},\tilde{A}}$ a.e.\;for some $\tilde{A}<A$.
Note that $\rho_{\tilde{R},\tilde{A}}$ has less steep boundary behavior than $\rho_{\tilde{R},A}$.
We need an auxiliary result.

\begin{lemma}\label{lem: simplify the constraints on tilde A}
Define $y_A\in [0,w_A]$ to solve
\begin{equation}
|Q'(y_A)| = \frac{1}{2}|Q'(w_A)|.
\label{eqn: def of r_*}
\end{equation}
Then there exists a universal $\tilde{c}$, which only depends on $\tilde{G}$, such that for all
$
\tilde{A}\in[ \frac{A}{1+\tilde{c}\tau},A ],
$
\begin{equation}
A\leq \tilde{A}(1-\tau \tilde{A}^2 Q''(y_A))
\label{eqn: 2nd constraint on A tilde}
\end{equation}
and
\begin{equation}
R- \frac{w_A}{A} \leq \tilde{R}-\frac{w_{\tilde{A}}}{\tilde{A}} - \frac{\tau A}{2}|Q'(w_A)|.
\label{eqn: first constraint on A tilde}
\end{equation}

\begin{proof}
We first consider \eqref{eqn: first constraint on A tilde}.
Recall that $Q$ is concave.
By the monotonicity of $w_A$, $Q(w_A)$ and $|Q'(w_A)|$ in $A$ due to Lemma \ref{lem: property of auxiliary function Q} and Lemma \ref{lem: property of auxiliary function Q_A},
\[
\begin{split}
|w_A-w_{\tilde{A}}||Q'(w_A)|\leq &\;Q(w_A)-Q(w_{\tilde{A}})\\
=&\; \frac{\tau A^2}{2}|Q'(w_A)|^2- \frac{\tau \tilde{A}^2}{2}|Q'(w_{\tilde{A}})|^2\leq \frac{\tau}{2}(A^2-\tilde{A}^2)|Q'(w_A)|^2.
\end{split}
\]
In the equality above, we used the definition of $w_A$ and $w_{\tilde{A}}$.
Hence,
$$
\frac{w_{\tilde{A}}- w_A}{\tilde{A}}\leq \frac{\tau}{2\tilde{A}}(A^2-\tilde{A}^2)|Q'(w_A)|\leq \frac{\tau A}{\tilde{A}}(A-\tilde{A})|Q'(w_A)|.
$$
In order that \eqref{eqn: first constraint on A tilde} holds, it suffices to have
$$
\frac{\tau A}{\tilde{A}}(A-\tilde{A})|Q'(w_A)|+w_A \frac{A-\tilde{A}}{\tilde{A}A}\leq \frac{\tau A}{2}|Q'(w_A)|,
$$
which is true if
\begin{equation}
\frac{\tilde{A}}{A}\geq \frac{1+\frac{\tau A^2 |Q'(w_A)|}{w_A}}
{1+\frac{3\tau A^2 |Q'(w_A)|}{2w_A}}.
\label{eqn: simplified constraint for tilde A}
\end{equation}
%
By Lemma \ref{lem: property of auxiliary function Q_A}, 
for all $A\geq 1$ and $\tau \leq 1$,
$$
\frac{A|Q'(w_A)|}{w_A}\geq \frac{|Q'(w_1)|}{w_1}\geq C(\tilde{G}).
$$

Therefore, there exists a universal constant $C(\tilde{G})$, such that $
\tilde{A}\geq \frac{A}{1+C(\tilde{G})\tau}$
implies \eqref{eqn: simplified constraint for tilde A}.
As a result, \eqref{eqn: first constraint on A tilde} is also true.

For \eqref{eqn: 2nd constraint on A tilde},
it suffices to show that $\tilde{A}|Q''(y_A)|\geq C(\tilde{G})$.
Since $|Q''|$ is an increasing function and $Q'(0) = 0$, by \eqref{eqn: def of r_*} and the result proved above,
$$
\tilde{A}|Q''(y_A)|\geq \frac{\tilde{A}}{A}\cdot A\cdot \frac{|Q'(y_A)|}{y_A} \geq \frac{\tilde{A}}{A}\cdot \frac{A|Q'(w_A)|}{2w_A}\geq C(\tilde{G}).
$$

Therefore, by suitably choosing $\tilde{c}$ that only depends on $\tilde{G}$, we have \eqref{eqn: 2nd constraint on A tilde} and \eqref{eqn: first constraint on A tilde} hold.
\end{proof}
\end{lemma}

Now we generalize Lemma \ref{lem: new density is dominated by another barrier}.
\begin{lemma}\label{lem: improved parameter}
Suppose $\tau \in (0,1]$.
Let $\rho_\dag$ and $\tilde{R}$ be defined as above, and let $\tilde{A}\in [1,A]$ satisfy \eqref{eqn: 2nd constraint on A tilde} and \eqref{eqn: first constraint on A tilde}.
Then $\rho_\dag\leq \rho_{\tilde{R},\tilde{A}}$.
\end{lemma}

\begin{proof}
Since $\tilde{A}\geq 1$, arguing as in Lemma \ref{lem: new density is dominated by another barrier}, it suffices to show $\tilde{W}(r)\geq W(r)$ for all $r\in [0,R]$, where

\begin{equation*}
\begin{split}
\tilde{W}(r) &:= q_{\tilde{R},\tilde{A}}(r+\tau|q_{R,A}'(r)|) - \frac{\tau}{2}|q_{\tilde{R},\tilde{A}}'(r+\tau|q_{R,A}'(r)|)|^2,\\
W(r) &:= q_{R,A}(r)-\frac{\tau}{2}|q_{R,A}'(r)|^2.
\end{split}
\label{eqn: def of W}
\end{equation*}

By definition, $W(r) = z_M$ for all $r\leq (R-w_A/A)$, but $W(r) < z_M$ if $r> (R-w_A/A)$.
On the other hand, by \eqref{eqn: first constraint on A tilde}, we have $\tilde{W}(r) = z_M$ on $[0,r_*]$ for some $r_*>(R-w_A/A)$.
Moreover, $W(R) = \tilde{W}(R) = 0$.

Let $\mathcal{S} = \{r\in[0,R]: \tilde{W}(r) < W(r)\}$.
Suppose $\mathcal{S} \neq \varnothing$.
We take $r_0 = \inf\mathcal{S}$.
Since $W$ and $\tilde{W}$ are continuous, $W(r_0) = \tilde{W}(r_0)$, i.e.,
\begin{equation}
Q(y(r_0))-\frac{A^2 \tau}{2}|Q'(y(r_0))|^2=
Q(\tilde{y}(r_0)) - \frac{\tilde{A}^2\tau}{2}|Q'(\tilde{y}(r_0)) |^2.
\label{eqn: equivalent form of equation for r}
\end{equation}
Here for brevity, we denote
\[
\begin{split}
y(r)&:=A(r-(R-w_A/A)),\\
\tilde{y}(r)&:= \tilde{A}\left(r+\tau|q_{R,A}'(r)|-\left(\tilde{R}-\frac{w_{\tilde{A}}}{\tilde{A}}\right)\right).
\end{split}
\]
Here it is also understood that $Q \equiv z_M$ on $(-\infty,0]$.
Since $\tilde{A}\leq A$, \eqref{eqn: equivalent form of equation for r} implies $y(r_0)\leq \tilde{y}(r_0)$, which further gives
\[
r_0+\tau A|Q'(y(r_0))|-\left(\tilde{R}-\frac{w_{\tilde{A}}}{\tilde{A}}\right)\geq r_0-\left(R-\frac{w_A}{A}\right).
\]
By \eqref{eqn: first constraint on A tilde}, we find $y(r_0)\geq y_A$, where $y_A$ is defined in Lemma \ref{lem: simplify the constraints on tilde A}.
Then by \eqref{eqn: 2nd constraint on A tilde} and the monotonicity of $Q''$, for all $r\geq r_0$,
\begin{equation}
A\leq \tilde{A}(1-\tau \tilde{A}^2 Q''(y(r))).
\label{eqn: properties of intersection point}
\end{equation}

We claim that $\tilde{W}'(r)\leq W'(r)$ for all $r\geq r_0$.
We calculate that
\[
\begin{split}
\tilde{W}'(r) &= q'_{\tilde{R},\tilde{A}}(r+\tau|q_{R,A}'(r)|)\left(1- \tau q_{\tilde{R},\tilde{A}}''(r+\tau|q_{R,A}'(r)|)\right)(1-\tau q_{R,A}''(r)),\\
W'(r) &= q_{R,A}'(r)(1-\tau q_{R,A}''(r)).
\end{split}
\]
By \eqref{eqn: def of q_RA}, $\tilde{W}'(r)\leq W'(r)$ is equivalent to
\[
\tilde{A} Q'(\tilde{y}(r)) (1- \tau \tilde{A}^2 Q''(\tilde{y}(r)))
\leq A Q'(y(r)),
\]
By virtue of \eqref{eqn: properties of intersection point} and the monotonicity of $Q'$ and $Q''$, it suffices to show that $y(r)\leq \tilde{y}(r)$ for all $r\geq r_0$.
Since this is true for $r= r_0$, we shall prove $y'(r)\leq \tilde{y}'(r)$ for all $r\geq r_0$, i.e.,
\[
A\leq\tilde{A}(1-\tau A^2 Q''(y(r))).
\]
This clearly follows from \eqref{eqn: properties of intersection point}.

Recall that $W(R) = \tilde{W}(R) = 0$ and $W(r_0)=\tilde{W}(r_0) $.
So if $r_0<R$, we must have $\tilde{W}'(r) = W'(r)$ for all $r\geq r_0$.
This implies $\tilde{W}(r) = W(r)$, which leads to a contradiction.
On the other hand, $r_0\neq R$, since $\mathcal{S}$ is open in $[0,R]$.

This proves $\mathcal{S} = \varnothing$, so $\tilde{W}(r)\geq W(r)$ on $[0,R]$.
This completes the proof.
\end{proof}

Combining these two Lemmas with Proposition \ref{comparison}, we argue as in Lemma \ref{lem: basic barrier functions} to conclude with the following result, from which Proposition \ref{prop_iteration} follows.

\begin{prop}
Suppose $\rho_0\in L^{\infty}(\Omega)\cap X$ satisfies \eqref{good_data_0}.
Fix $\tau\in (0,1]$.
Let $\{\rho^n\}$ be the sequence of densities obtained by the discrete scheme \eqref{eq:mms} starting from $\rho_0$.
Take $\rho^+ = \|\rho_0\|_{L^\infty(\Omega)}$ and let $R_*$ and $c_*$ be defined as in Lemma \ref{lem: basic barrier functions}.

With $A\geq 1$, let $q_{R,A}$ and $\rho_{R,A}$ be defined in \eqref{eqn: def of q_RA} and \eqref{eqn: precise def of rho_RA}, respectively.
Suppose $\rho_0\leq \rho_{R_0,A_0}$ for some $A_0\geq 1$ and $R_0\geq w_{A_0}/A_0+1$.
With $\tilde{c}$ defined in Lemma \ref{lem: simplify the constraints on tilde A}, define $\{A_n\}_{n = 0}^\infty$ and $\{R_n\}_{n = 0}^\infty$ as follows:
\begin{align*}
A_{n} = \max\left\{1, \frac{A_0}{(1+\tilde{c}\tau)^n}\right\},\quad R_{n} = R_{n-1}+\tau A_{n-1}|Q'(w_{A_{n-1}})|.
\end{align*}
Then $\rho^n \leq \rho_{R_n,A_n}$
for all $n \in \N$ which satisfies $B_{R_n}\subset \Omega$.
In particular, $\spt\rho^n \subset \overline{B_{R_n}}$.
\end{prop}

\bibliography{tumor_refs}
\bibliographystyle{amsalpha}

\end{document}